\documentclass[a4paper, 11pt]{amsart}
\usepackage{amsmath, amssymb, amscd, amsthm, amsfonts, tikz, tikz-cd}
\usepackage{graphicx}
\usetikzlibrary{arrows.meta,positioning}
\tikzset{>=Stealth, dot/.style={circle,fill,inner sep=1.3pt}}
\usetikzlibrary{decorations.markings}
\tikzset{double line with arrow/.style args={#1,#2}{decorate,decoration={markings,%
mark=at position 0 with {\coordinate (ta-base-1) at (0,1pt);
\coordinate (ta-base-2) at (0,-1pt);},
mark=at position 1 with {\draw[#1] (ta-base-1) -- (0,1pt);
\draw[#2] (ta-base-2) -- (0,-1pt);
}}}}
\usepackage{stmaryrd}
\usepackage{comment}
\usepackage{mathrsfs}
\usepackage{hyperref}
\usepackage[T1]{fontenc}
\usepackage[british,french]{babel}
\usepackage[all]{xy}
\usepackage[margin=24.4mm]{geometry}
\newcommand{\MG}[1]{\textcolor{blue}{[MG: \footnotesize #1]}}
\newcommand{\ZJ}[1]{\textcolor{red}{[ZJ: \footnotesize #1]}}
\newcommand{\todo}[1]{\textcolor{violet}{[To do: \footnotesize #1]}}

\title{A Categorification of Subword Complexes and Its Hall Algebra}

\author{Mikhail Gorsky}
\address{Institut Camille Jordan UMR 5208, Université Jean Monnet, CNRS, Centrale Lyon, INSA Lyon, Université Claude Bernard Lyon 1, 20, rue Annino, 42023, Saint-Étienne, France;
\newline and Universit\"at Hamburg, Fachbereich Mathematik, Bundesstraße 55, 20146 Hamburg, Germany}
\email{mikhail.gorskii@univ-st-etienne.fr}

\author{Zijun Li}
\address{Institut Camille Jordan UMR 5208, Université Jean Monnet, CNRS, Centrale Lyon, INSA Lyon, Université Claude Bernard Lyon 1, 20, rue Annino, 42023, Saint-Étienne, France}
\email{zijun.li@univ-st-etienne.fr}

\date{\selectlanguage{british}\today}

\newtheorem{theorem}{Theorem}[section]
\newtheorem{lemma}[theorem]{Lemma}
\newtheorem{proposition}[theorem]{Proposition}
\newtheorem{corollary}[theorem]{Corollary}

\theoremstyle{definition}
\newtheorem{definition}[theorem]{Definition}

\newtheorem{example}[theorem]{Example}

\theoremstyle{remark}
\newtheorem{remark}[theorem]{Remark}

\newcommand{\V}{\operatorname{V}}
\newcommand{\F}{\operatorname{F}}
\newcommand{\rx}{\operatorname{r}}
\newcommand{\bx}{\operatorname{b}}

\newcommand{\SC}{\mathcal{SC}}

\newcommand{\spa}{\mathscr{I}}
\newcommand{\Irr}{\operatorname{Irr}}
\newcommand{\Ext}{\operatorname{Ext}}
\newcommand{\Vect}{\operatorname{Vect}}

\newcommand{\CC}{\mathcal{C}}
 
\newcommand{\spanx}{\operatorname{span}} 
\newcommand{\id}{\operatorname{id}}

\newcommand{\CN}{\mathbb{C}[N]}

\newcommand{\Mor}{\operatorname{Mor}}

\addto\captionsfrench{}
\addto\captionsfrench{}

\addto\captionsfrench{}
\begin{document}


\begin{abstract}
    Bergeron and Ceballos defined a Hopf algebra structure on equivalence classes of subword complexes. We introduce a category of subword complexes, endow it with a proto-exact-like structure, and show that the corresponding dual Hall Hopf algebra is isomorphic to the algebra of Bergeron--Ceballos. We prove that the full subcategory of root-independent objects is  proto-abelian in the sense of Dyckerhoff. We give a categorical lift of flips in subword complexes. 


    We consider a version of a category of formal direct sums of subobjects for a root-independent subword complex and interpret it in terms of quivers. If the corresponding quiver is a tree, the category is endowed with a proto-exact structure. We show that its Hall algebra is isomorphic to the Hall algebra of the category of representations of this  quiver over $\mathbb{F}_1$. Under certain conditions, a flip corresponds to changing a proto-exact structure while keeping the category the same up to isomorphism, which corresponds to a non-trivial automorphism of the Hall algebra. In type $A$, this leads to a realization of the nilpotent part of the universal enveloping algebra and its automorphisms.

\end{abstract}

\maketitle


\section*{Introduction}

Subword complexes are simplicial complexes introduced by Knutson and Miller in \cite{subword, KM2005}. 
Originally motivated by problems of Gr\"obner degenerations of matrix Schubert varieties, these complexes has since been related to a number of different topics in representation theory, algebraic combinatorics, total positivity, cluster algebras, to name a few. A subword complex is associated with a pair $(Q, w)$ of an element $w$ of a Coxeter group $W$ and a word $Q$ in the alphabet of simple reflections. The effects on the complex of applying natural operations to $Q$ while preserving $w$ have been studied e.g. in \cite{cls, gorsky_edge, G_subword_3}. A subword complex is either a combinatorial sphere, or a ball. In the spherical case, dual subword complexes describe adjacency of strata in stratifications of brick varieties \cite{escobar}. These are subvarieties of Bott-Samelson type varieties and provide smooth compactifications of a family of so-called braid varieties \cite{CGGS2}. The latter generalize open Richardson and positroid varieties, are isomorphic to spectra of locally acyclic cluster algebras \cite{CGGLSS,GLSB, CGGSSBS}, and are related to invariants of smooth and Legendrian realizations of closures of positive braids \cite{trinh2021hecke,casals_gao}.
 Brick varieties admit Hamiltonian torus actions, with moment polytopes being brick manifolds introduced in \cite{pilaud_santos_12,pilaud_stump_15}. Under a fairly restrictive condition of root-independency, a brick variety is precisely the toric variety of a brick manifold \cite{escobar}; in this case, the latter provides a geometric realization of the dual subword complex \cite{pilaud_stump_15}. Cluster complexes of finite type cluster algebras are isomorphic to certain root-independent subword complexes, and corresponding brick manifolds realize generalized associahedra \cite{cls, pilaud_stump_15}. The non-spherical case has been studied in \cite{JahnStump}. In a different direction, the interior dual block complexes of subword complexes have been conjectured to be homeomorphic to the fibers of the
Chevalley exponentiation maps to totally nonnegative spaces \cite{DHM}.


Bergeron and Ceballos introduce the Hopf algebra of subword complexes in \cite{BC}. An element of this Hopf algebra is (an equivalence class of) a subword complex equipped with a maximal face. With any subword complex, there is an associated vector space generated by a configuration of certain real roots of the corresponding Coxeter group. The multiplication in the Hopf algebra of subword complexes is induced by the direct sum of vector spaces. The comultiplication is given by the decomposition of a vector space into the direct sum of smaller vector spaces. \par

Over the last fifteen years, various Hopf algebras appearing in combinatorics, representation theory, study of mapping class groups, etc., have been considered from a unifying perspective of 2-Segal spaces \cite{DK}, also known as decomposition spaces \cite{GCKT}. These are simplicial spaces organised in such way that their $1$-simplices
can be arranged into an algebra and a coalgebra structures, with $2$-simplices inducing the operations and the interplay between $1$-, $2$-, and $3$-simplices governing the (co)associativity and, in some cases, compatibility of the multiplication and comultiplication. We refer the reader to the original sources and to more recent surveys \cite{PRO, CY}. One source of such simplicial spaces comes from variations of the Waldhausen $S^\bullet$-construction applied to categories with certain extra structure giving some version of Yoneda-style first extensions. Dyckerhoff and Kapranov \cite{DK} introduced a very general notion of a suitable structure, so-called proto-exact categories (and their higher categorical analogues, proto-exact $\infty$-categories) giving a non-additive generalization of Quillen exact structures. Proto-exact categories appear in various contexts, including study of pointed sets and matroids, normed Euclidean spaces, Arakelov geometry, and, in additive case, representation theory and algebraic geometry. 

Given a proto-exact category, the (co)algebra associated with is via the 2-Segal space machinery of \cite{DK} is called its Hall (co)algebra. For abelian categories, such algebras have been introduced by Ringel in a series of papers starting from \cite{R1}. This (co)algebra is a vector space with basis given by isomorphism classes of objects in the category and operations defined in terms of first extension groups. The prime example is given by the Hall algebra of the category of nilpotent representations of a finite quiver $Q$ over a finite field $\mathbb{F}_q$. Ringel proved that its naturally defined subalgebra is isomorphic, up to a twist of multiplication, to the nilpotent part of the Drinfeld-Jimbo quantum group corresponding to $Q$, with the quantization parameter specialized at $\sqrt{q}$. There exists an extended version which further carries a Hopf algebra structure and is related in the similar way to the Borel part of the quantum group. We refer the reader to \cite{Hu, schiffmann1, schiffmann2, schiffmann3} for surveys on this construction and many further examples of Hall algebras.

 

In this paper, we 
introduce a categorification $\mathscr{C}$ of subword complexes, partially motivated by these considerations.  
An object in $\mathscr{C}$ is a subword complex equipped with a maximal face. Morphisms are given by equivalent pairs of subobjects. Then we show that the Hopf algebra of subword complexes can be interpreted as a dual Hall Hopf algebra $\mathcal{H}(\mathscr{C})$ defined on the set of isomorphism classes of objects in $\mathscr{C}$. 
The category $\mathscr{C}$ is a non-additive cofinitary category, meaning that the morphism set between any two objects in $\mathscr{C}$ is a finite set and any object has only finitely many (isomorphism classes) of subobjects. There is a natural forgetful functor $\mathscr{C} \to \Vect_{\mathbb{R}}$, and the $\Ext^1$-structure defining the  comultiplication is defined by pulling back the (split) exact structure on $\Vect_{\mathbb{R}}$ along this functor. We note that this structure is not proto-exact, but the Hall comultiplication defined in the same way as for proto-exact categories is still coassociative. It is moreover compatible with the multiplication defined via a symmetric monoidal structure playing the role of direct sum.

We further realize certain Hopf subalgebras of the Bergeron-Ceballos algebra as dual Hall algebras of full subcategories of $\mathscr{C}$.
In particular, we prove that the full subcategory $\mathscr{D}$ of $\mathscr{C}$ formed by the root-independent complexes is cofinitary proto-abelian in the sense of \cite{DK, PRO}, and the corresponding Hopf subalgebra fits into the framework of (dual) Hall coalgebras of such categories defined in \emph{loc. cit.}.
When we focus on objects whose associated Coxeter group is finite, we have a proto-abelian finitary (and still cofinitary) category $\mathscr{D}_{fin}$. 
The multiplication defined by all exact sequences in the proto-abelian structure on $\mathscr{D}_{fin}$ defines a Hall algebra of $\mathscr{D}_{fin}$. The comultiplication coming from $X \oplus Y$ is compatible with this multiplication, and thus defined Hall Hopf algebra is dual to the subalgebra of the Bergeron-Ceballos algebra on root-indedependent subword complexes for finite Coxeter groups.
\par

Given an object $X\in \mathscr{D}$, we construct a category $\mathcal{L}_X$ using a version of the span category of the subcategory of $\mathscr{D}$ given by subobjects of $X$. From the configuration of roots associated with $X$, we define a quiver which we call the root configuration quiver of $X$ and denote 
$\Gamma_X$. When the subword complex associated with $X$ is a realization of a $c$-cluster complex of simply-laced type as in \cite{cls} and the maximal face is the leftmost one, $\Gamma_X$ is just the corresponding Dynkin quiver. We construct a subquiver category $\mathcal{S}_X$ whose objects are disjoint unions of subquivers of $\Gamma_X$ and show that $\mathcal{L}_X$ can be viewed as a subcategory of $\mathcal{S}_X$. When $\Gamma_X$ is a tree, the partial order on the set of vertices of $\Gamma_X$ induces a proto-exact structure on $\mathcal{S}_X$. We prove that the Hall algebra of this proto-exact category is isomorphic to the Hall algebra of the category of representations of $\Gamma_X$ over $\mathbb{F}_1$. 
The latter category and its Hall algebra were introduced by Szczesny in \cite{quiverf1} and studied extensively e.g. in \cite{JunSistko23, JunSistko, FRY}. We note that while the Hall algebras are isomorphic, the categories are not equivalent, even if we forget about their proto-exact structures.
As an application, we show that in type $A$ this isomorphism induces a realization of the universal enveloping algebra using the Hall algebra on $\mathcal{S}_X$ given by this proto-exact structure. \par

There are some natural combinatorial operations on subword complexes. We prove that under certain conditions, the flip operation between two adjacent maximal faces corresponds to the reflection of root configuration quivers. The category $\mathcal{S}_X$ stays the same, up to canonical isomorphism, but the proto-exact structure does depend on the orientation and is thus changed. As is studied in the additive case in \cite{FangGorsky}, one can consider Hall algebras of different proto-exact structures on the same category. For $\mathcal{S}_X$, it follows from results of \cite{JunSistko23} that the change of structure induced by such a flip (i.e. by a reflection) induces a non-trivial automorphism of the Hall algebra and so of the associated universal enveloping algebra. \par



The category of subword complexes is also closely related to galleries in the Coxeter complex. For the theory of folded galleries in Coxeter complexes, see, for example, \cite{Schwer}. Any object $X$ in $\mathscr{C}$ corresponds to a unique folded gallery on the Coxeter complex of $W_X$, and objects are isomorphic if and only if the corresponding galleries are identified by an isomorphism of Coxeter systems. The category $\mathscr{C}$ can thus be seen as a category of folded galleries. In the companion paper of the second-named author \cite{Zijun},
it is proved that any flat of an object corresponds to the projection of the folded gallery into a sub-Coxeter complex. \par
In particular, the comultiplication formula in the Hall algebra of the subquiver category $\mathcal{S}_X$ defined in Section~\ref{sec:categories_of_subobjects_and_subquivers} can be translated to a comultiplication formula for folded galleries in the Coxeter complex. In the same paper, it is shown that in type $A$ this comultiplication coincides with the comultiplication of the coordinate ring of the corresponding unipotent group $\CN$ using the Mirkovic-Vilonen (MV) polytope model. \par
The MV polytope model parameterizes a crystal basis on $\CN$ using the crystal structure \cite{JOEL}. In particular, we can choose a set of fundamental MV polytopes which parameterizes a generating set of $\CN$ as an algebra \cite{BAI}. These MV polytopes correspond to certain folded galleries in the type $A$ Coxeter complex. Via these relations, we can map certain objects in $\mathscr{C}$ with Coxeter group of type $A$ to polynomials in the coordinate ring $\CN$. See \cite{MVHD}, \cite{GL} and \cite{BDKT} for more details on MV polytopes and the crystal theory.\par

\subsection*{Organisation of the paper}
In Section~\ref{sec:category_C}, we recall basic definitions on subword complexes and construct the categorification. In Section~\ref{sec:root-independednt_subcat}, we recall some definitions about the proto-abelian category and prove that the root independent subcategory is proto-abelian. In Section~\ref{sec:categories_of_subobjects_and_subquivers}, we define the root configuration quiver and the subquiver category. When the root configuration quiver $\Gamma_X$ is a tree, 
we prove that the partial order induced by $\Gamma_X$ gives a proto-exact structure on $\mathcal{S}_X$. Then 
we show the isomorphism of Hall algebra between $\mathcal{H}(\mathcal{S}_X)$ and $\mathcal{H}(Rep(\Gamma_X,\mathbb{F}_1))$. Then we get a non-trivial isomorphism of the Hall algebra using the flips of the subword complex. At the end of Section~\ref{sec:categories_of_subobjects_and_subquivers}, we show that this gives a realization of the universal enveloping algebra in type $A$ and give some examples.


\section{Categorification of subword complexes} \label{sec:category_C}

We refer the reader to \cite{bb_coxeter} and \cite{AbramenkoBrown2008} for the basic background on Coxeter groups, root systems, and Coxeter complexes. 

\subsection{Subword complexes and flats}
Given an arbitrary Coxeter group $W$ of rank $n$ with generators $\{s_i\}_{i\in [n]}$, a word $Q=s_{i_1}\cdots s_{i_m}$ and an element $\pi\in W$, the subword complex $\SC_W(Q,\pi)$ is the simplicial complex whose faces are subsets $I$ of $[m]$ such that the subword of $Q$ with positions at $[m]\backslash I$ contains a reduced expression of $\pi$. \par
We consider the set of all quadruples $\tilde{\mathcal{C}}=\{(W,Q,\pi ,I)\}$, where $W$ is an arbitrary Coxeter group and $I$ is a maximal face of $\SC_W (Q,\pi)$. We denote the length of the word $Q$ by $l(Q)$. The maximal face $I$ can be viewed as a subset of $[n]=\{1,\cdots,n\}$. We use $k$ to denote a field containing $\mathbb{Z}$ in this paper. \par

In \cite{BC}, Bergeron and Ceballos give an equivalence relation for such quadruples:

\begin{definition}
\label{def:equivalence}
Two quadruples $X=(W,Q,\pi ,I)$ and $X=(W',Q',\pi' ,I')$ in $\tilde{\mathcal{C}}$ are equivalent if and only if there exists a group isomorphism $\phi$ from $W$ to $W'$ given by a reassignment of generators of $W$ such that $\phi(Q)=Q'$ up to commutation of consecutive commuting letters, $\phi(\pi)=\pi'$ and $I'$ is the positions corresponding to the positions of $I$ in $\phi(Q)$, up to the performed commutations of the form $\ldots s_i s_j \ldots \leftrightarrow \ldots s_j s_i \ldots$, with $i \neq j, s_is_j = s_js_i$. We denote this equivalence relation by $X\sim X'$.
\end{definition}

\begin{remark} \label{rem:no_swapping_identical_letters}
In Definition~\ref{def:equivalence}, the convention is that swapping two consecutive identical letters is not an equivalence. For example, the quadruples $(S_2, ss, s, \{1\})$ and $(S_2, ss, s, \{2\})$ are not equivalent. This can be interpreted via root configurations defined below: these two quadruples have different root configurations. More conceptually, they correspond to different folded galleries.
\end{remark}

 For any quadruple $X$ in $\tilde{\mathcal{C}}$, we usually use the subscript $X$ to denote the component corresponding to $X$ in the quadruple. In other words, $X=(W_X,Q_X,\pi _X,I_X)$. The real vector space spanned by the roots of $W_X$ is denoted by $V_X$. The length of the word $Q_X$ is denoted by $n_X$. \par
 
For any $l\in [n_X]$, we have an associated root 
 \begin{align}
\rx_X(l):=\prod_{j\in [l-1] \backslash I_X}s_{i_j}(\alpha_{i_l}). \label{eq:root function}    
 \end{align}
We get a function $\operatorname{r}_X$ from $[n_X]$ to the set of roots associated with $W_X$ by Equation~\eqref{eq:root function}, which is called the \textbf{root function} of $X$. Let $R(X):=\{\rx _X(i)\}_{i\in I_X}$ denote the \textbf{root configuration} of $X$ and $r(X):=\{\rx_X(i)\}_{i\in [n_X]}$ denote the set of all such roots of $X$. A quadruple $X$ is \textbf{irreducible} if $V_X=\operatorname{span}R(X)$. For any subset $H$ of $[n_X]$, we denote by $\V(H)$ the vector space spanned by $\{r_X(i)\}_{i\in H}$.  \par

 We denote the collection of irreducible quadruples by $\CC\subseteq \tilde{\mathcal{C}}$. The vector space spanned by the set of equivalence classes of irreducible quadruples can be equipped with a connected graded commutative and cocommutative Hopf algebra structure by \cite[Theorem 3.6]{BC}.
 \par

In this paper, we define morphisms between two quadruples to upgrade the set $\mathcal{C}$ to a category and show that the Hopf algebra of subword complexes in \cite{BC} can be viewed as a dual Hall algebra $\mathcal{H}(\mathscr{C})$ of this category. \par

In \cite{BC}, Bergeron and Ceballos introduce the concept of a flat: 

\begin{definition} \label{def:flats}
    Given a quadruple $X\in \tilde{\mathcal{C}}$ and a subspace $V$ of $V_X$, the subset $\F(V)\subseteq[n_X]$ defined by the equation \begin{align}
        \{\rx_X(i)\}_{i\in \F(V)}=r(X) \cap V.   \nonumber
    \end{align}
    is called the \textbf{flat} associated to the subspace $V$. \par
    We say that a flat $F'$ is a \textbf{subflat} of another flat $F$ if $F'\subseteq F$.
\end{definition}
 Notice that $\F(\V(F))=F$ for any flat $F$ and $\V(\F(V))=V$ for any subspace $V$ of $V_X$. \par
 
Given a flat $F$ of $X$, the subspace $\V(F)$ of $V_X$ contains a natural root system
\[
\Phi_F = \Phi_F^+ \sqcup \Phi_F^- 
\]
where $\Phi_F$, $\Phi_F^+$, $\Phi_F^-$ are the restrictions of $\Phi_X$, $\Phi_X^+$, $\Phi_X^-$ to $\V(F)$ respectively. The fact that the intersection of a root system and a subspace is again a root system with simple roots contained in $\Phi^+$ is a non-trivial result by Dyer in \cite{Dyer} and by Deodhar in \cite{Deodhar}. Suppose that $Q_X=s_{i_1}\cdots s_{i_m}$. Set $F=\{j_1,\cdots ,j_r\}$ where $j_l< j_{l+1}$ for any $l<r$.  Define $\beta _F = (\beta_1, \ldots, \beta_{r})$ as the list of roots
\[
\beta_k := \prod_{j\in[B_k]}s_{i_j} (\alpha_{i_{j_k}}),
\]
where $B_k := ([j_k-1] \setminus I_X)\setminus F$ is the set of positions on the left of $j_k$ in the complement of $I_X$ which are not in $F$. \par

We denote the order preserving bijection from $F$ to $[n_{X_F}]$ by $\bx_F$. The roots associated to a quadruple are well preserved by the following proposition in \cite{BC}:

\begin{lemma}\cite[Lemma 2.10 and Theorem 2.11]{BC} \label{lem:flats_quadruples_BC}
     Any flat $F$ of an quadruple $X\in \tilde{\mathcal{C}}$ induces a quadruple $X_F\in \tilde{\mathcal{C}}$, such that $V_{X_F}=\V(F)$ and $n_{X_F}=|F|$. Simple roots of $\Phi_{X_F}$ are given by $\beta_F$. We have $\rx_X(i)=\rx_{X_F}(\operatorname{b}_F(i))$ for any $i\in F$. 
\end{lemma}

\begin{definition}
    For a quadruple $X\in \CC$, a flat $F$ of $X$ is called \textbf{irreducible} of $X$ if $V_F = \operatorname{span}R(I_X\cap F)$, in other words, if $X_F\in \CC$.
    
    Given an irreducible flat $F'$ of $X$, we say that $F'$ is a \textbf{complementary flat} of $F$ in $X$ if $\V(F)\oplus \V(F')=V_X$. We denote a complementary flat by $F^{\perp}$.
       We denote the set of irreducible flats of $X$ by $\operatorname{Irr}(X)$. In particular, the empty flat is viewed as an irreducible flat.
\end{definition}
    

    

Notice that a flat $F$ of $X\in \CC$ is irreducible if and only if $F=\F(\V(F\cap I_X))$.

\begin{remark}

 Suppose that $X\in \mathcal{C}$, then $R(X)$ generates $V_X$. For any flat $F$ of $X$ such that $V_{X_F}=\V(F)$, we can always add some vectors from $R(X)\backslash \{\rx_X(i)| i\in F\cap I_X\}$ to $ \{\rx_X(i)| i\in F\cap I_X\}$ to obtain a generating set of $V_X$. Thus the flat $F$ always has a complementary. 

  Notice however that an irreducible flat can have many different complementary flats, which can even have different cardinalities, see Examples~\ref{S3} and ~\ref{S4} below.
\end{remark}

\subsection{The construction of the category of subword complexes \texorpdfstring{$\mathscr{C}$}{C}}

\begin{definition} \label{def:morphisms_in_C}
    Given $X,Y\in \CC$, we define the morphism set by 
   \begin{align}
       \operatorname{Mor}(X,Y)= \{(F_1,F_2)\}, \nonumber
   \end{align}
   where $(F_1,F_2)$ ranges over all pairs such that $F_1$ is an irreducible flat of $X$, $F_2$ is an irreducible flat of $Y$ and $X_{F_1}\sim Y_{F_2}$. We denote the pair given by the isomorphism between trivial flats $(\emptyset,\emptyset)$ by $0_{X,Y}$.
\end{definition}

To simplify the following construction, we introduce some auxiliary notations. Fix two quadruples $X,Y\in \CC$. For any $(F_1,F_2)\in \Mor(X,Y)$, the equivalence relation $X_{F_1}\sim Y_{F_2}$ gives a bijection $\gamma_{F_1,F_2}$ from $F_1$ to $F_2$ by Definition~\ref{def:equivalence}. We denote the linear isomorphism from $\V(F_1)$ to $\V(F_2)$ induced by $\gamma_{F_1,F_2}$ by $l_{F_1,F_2}$. Notice that if $(F_1,F_2)\in \Mor(X,Y)$, we have $(F_2,F_1)\in \Mor(Y,X)$ automatically. \par

\begin{definition} 
    Given $X,Y,Z\in \CC$, $(F_1,F_2)\in \Mor(X,Y)$ and $(F_3,F_4)\in \Mor(Y,Z)$, the composition of morphisms is defined as follows: \begin{align}
        (F_3,F_4)\circ (F_1,F_2)= 
(\F(l_{F_2,F_1}(\V(F_2\cap F_3\cap I_Y))), \F(l_{F_3,F_4}(\V(F_2\cap F_3 \cap I_Y)))). \label{eq:composition_C}
\end{align}
\end{definition}
 
\begin{theorem}
    The collection $\mathscr{C}=\bigl(\CC, \Mor, \circ\bigr)$ forms a category with zero object $\mathbf{0}=(W_{triv},\emptyset, \emptyset,\emptyset)$ and zero morphisms. This category $\mathscr{C}$ is called the \textbf{subword complex category}.
\end{theorem}
\begin{proof}
    Given $X,Y,Z\in \CC$, $a=(F_1,F_2)\in \Mor(X,Y)$ and $b=(F_3,F_4)\in \Mor(Y,Z)$, the intersection $F_2\cap F_3\cap I_Y$ gives an irreducible flat $\F(\V(F_2\cap F_3\cap I_Y))$ of $Y$, which is a subflat of $F_2$ and $F_3$. Since we have $X_{F_1}\sim Y_{F_2}$ and $Y_{F_3}\sim Z_{F_4}$ by the definition of morphisms, Equation~\eqref{eq:composition_C} is well defined. \par
    The identity $\id_X\in \Mor (X,X)$ is given by $([n_X],[n_X])$. \par
    To prove the associativity, we fix an object $T\in \CC$ and a morphism $c=(F_5,F_6)\in\Mor(Z,T)$. We need to prove $(c\circ b)\circ a=c\circ (b\circ a)$. The morphism $b$ gives an order preserving bijection $\gamma$ between $F_3\cap I_Y$ and $F_4\cap I_Z$. The bijection between $F_2\cap\gamma^{-1}(F_4\cap F_5\cap I_Z)$ and $\gamma(F_2\cap F_3 \cap I_Y)\cap F_5$ induced by $\gamma$ gives the equality $(c\circ b)\circ a=c\circ (b\circ a)$ via the equivalence $Y_{F_3}\sim Z_{F_4}$.
\end{proof}

\begin{proposition} \label{prop:morphisms_in_C_finite}
    For any two quadruples $X,Y\in \CC$, the morphism set $\Mor(X,Y)$ is finite.
\end{proposition}

\begin{proof}
    Any flat of $X$ is a subset of $[n_X]$.
\end{proof}

Now we describe monomorphisms and epimorphisms in $\mathscr{C}$.

\begin{proposition} \label{prop:epi_mono_in_C}
    Given $X,Y\in \CC$ and a morphism $f=(F_1,F_2)\in \Mor(X,Y)$, $f$ is a monomorphism if and only if $F_1=[n_X]$ and $f$ is an epimorphism if and only if $F_2=[n_Y]$.
\end{proposition}

\begin{proof}
    Suppose that $f$ is a monomorphism. If $F_1\neq [n_X]$, the vector space $\V(F_1)$ is a proper subspace of $V_X$. Choose a complementary flat $F_1^{\perp}$ of $F_1$ in $X$. Since $\V(F_1)$ is a proper subspace, the complementary space $\V(F_1^{\perp})$ is non zero. The irreducibility of $X$ implies that $F_1^{\perp}$ is not zero. \par
    
    The order preserving bijection from $[n_{X_{F_1^{\perp}}}]$ to $F_1^{\perp}$ gives a nonzero morphism $h_1=([n_{X_{F_1^{\perp}}}],F_1^{\perp})$ in $\Mor(X_{F_1},X)$.
    By the definition of a complementary flat, we have $\V(F_1^{\perp})\cap V(F_1)=\{\mathbf{0_{V_X}}\}$. We have \begin{align}
        f\circ h_1=0_{X_{F_1^{\perp}},Y}=f \circ 0_{X_{F_1^{\perp}},X}. \nonumber
    \end{align}
    The assumption that $f$ is a monomorphism implies that $h_1=0_{X_{F_1^{\perp}},X}$, a contradiction. Thus  $F_1= [n_X]$. \par

   Conversely, suppose that $F_1= [n_X]$. The intersection of the vector space $\V(F)$ and $\V(F_1)=V_X$ is $\V(F)$ itself for any irreducible flat $F$. For any $Z\in \CC$ and $h_i=(H_{i,1},H_{i,2})\in \Mor(Z,X)$ where $i\in \{1,2\}$, we have $f\circ h_i=(H_{i,1}, \F(l_{F_1,F_2}(\V(H_{i,2}))))$. \par
   
   As a result, the equation $f\circ h_1=f\circ h_2$ implies that $H_{1,1}=H_{2,1}$ and $\F(l_{F_1,F_2}(\V(H_{1,2})))=\F(l_{F_1,F_2}(\V(H_{2,2})))$. We have 
   \begin{align}
       l_{F_1,F_2}(\V(H_{1,2}))=\V(\F(l_{F_1,F_2}(\V(H_{1,2}))))=\V(\F(l_{F_1,F_2}(\V(H_{2,2}))))=l_{F_1,F_2}(\V(H_{2,2})). \nonumber
   \end{align}
   Since $l_{F_1,F_2}$ is a linear isomorphism, we have $V(H_{1,2})=V(H_{2,2})$. By Definition~\ref{def:flats}, we have \begin{align}
       H_{1,2}=\F(\V(H_{1,2}))=\F(\V(H_{2,2}))=H_{2,2}. \nonumber
   \end{align}
   Thus the two morphisms $h_1$ and $h_2$ are equal in $\Mor(Z,X)$, which implies that $f$ is a monomorphism. \par

   The proof for epimorphisms is completely similar to the proof for monomorphisms by the symmetry between the morphism sets $\Mor(X,Y)$ and $\Mor(Y,X)$.
\end{proof}

We have the following observation: 

\begin{proposition} \label{prop:iso_in_C}
    Given $X,Y\in \CC$, $X$ is isomorphic to $Y$ in $\mathscr{C}$ if and only if $X\sim Y$.
\end{proposition}
\begin{proof}
    If $X\simeq Y$ in $\mathscr{C}$, we have an isomorphism $f: X\to Y$ by definition. $f$ is both a monomorphism and an epimorphism. By Proposition~\ref{prop:epi_mono_in_C}, we have $f=([n_X,n_Y])$. Since $X_{[n_X]}=X$ and $Y_{[n_Y]}=Y$, we have $X\sim Y$ by Definition~\ref{def:morphisms_in_C}.
\end{proof}

The morphisms in the category $\mathscr{C}$ can thus be interpreted either as spans of monomorphisms  or as cospans of epimorphisms of quadruples, 
with composition given  essentially by pullbacks/pushouts in a reasonable sense.

 Given any monomorphism $X\overset{f}{\hookrightarrow}Y$ in $\mathscr{C}$, there exists an irreducible flat $F$ of $Y$ such that $X \sim Y_F$. Given any epimorphism $X\overset{f}{\twoheadrightarrow}Y$ in $\mathscr{C}$, there exists an irreducible flat $F$ of $X$ such that $Y \sim X_F$.

\begin{definition} \label{def:subquotients_in_C}
    The object $X_F$ in $\mathscr{C}$ with the \textbf{canonical embedding} $X_F\overset{i_F}{\hookrightarrow}X$ in $\mathscr{C}$ represents a \textbf{subobject} of $X$, for any irreducible flat $F$ of $X$, where $i_F=([n_{X_F}],F)$. The set of subobjects of an object $X\in \mathcal{C}$ is parameterized by $\Irr(X)$. Symmetrically, we have a \textbf{canonical retraction} $X\overset{r_F}{\twoheadrightarrow}X_F$, where $r_F=(F,[n_{X_F}])$. This represents a \textbf{quotient object} of $X$.
\end{definition}

 Having introduced canonical embeddings and retractions, we can now interpret morphisms in $\mathscr{C}$ as certain linear maps of vector spaces. More precisely, we have the following forgetful functor $\mbox{For}: \mathscr{C} \to \Vect_{\mathbb{R}}$: on objects, $\mbox{For}: X \mapsto V_X$, and $(F_1, F_2) =  X \overset{r_{F_1}}\to X_{F_1} \overset\sim\to Y_{F_2} \overset{i_{F_2}}\to Y \in \Mor(X, Y)$ is sent by $\mbox{For}$ to the composition $i \circ l_{F_1, F_2} \circ p$, where $p: V_X \to V_{X_{F_1}}$ is the projection and $i: V_{Y_{F_2}} \to V_Y$ is the natural embedding. It is straightforward to check that this does indeed define a covariant functor $\mbox{For}: \mathscr{C} \to \Vect_{\mathbb{R}}$. Essentially by definition, this forgetful functor is faithful (but not full). It thus comes with no surprise that it reflects monomorphisms and epimorphisms.




To introduce the dual Hopf Hall algebra, we need to define admissible sequences for $\mathscr{C}$.

\begin{definition}
    A sequence of composable morphisms $A \overset{f}{\hookrightarrow} B \overset{g}{\twoheadrightarrow} C$ in $\mathscr{C}$ is called an \textbf{admissible sequence} if $f$ is a monomorphism, $g$ is an epimorphism and there exists an irreducible flat $F$ of $B$ with a complementary flat $F^{\perp}$ such that $A\sim B_F$ and $C\sim B_{F^{\perp}}$. Equivalently, it is admissible if and only if it is sent by the functor $\mbox{For}$ to a short exact sequence of vector spaces (which is automatically split).

\end{definition}

\begin{remark} \label{rem:not_weak_(co)kernels}
A word of caution: in thus defined admissible sequences, we have $h \circ f = 0$, but it is not true in general that $f$ is the kernel of $g$ or that $g$ is the cokernel of $f$. Furthermore, it is not even true that $g$ is a weak cokernel of $f$ (i.e. $g \circ f = 0$ and for each $g': B \to C'$, there exist a(not necessarily unique) morphism $l: C \to C'$ such that $g' = l \circ g$) or that $f$ is a weak kernel of $g$.  A weak (co)kernel of a given morphism, if exists, is unique up to non-unique isomorphism. In $\mathscr{C}$, the isomorphism class of an end term of an admissible sequence is not determined by the morphism between the other two terms. This is illustrated by the following two examples. 
\end{remark}





\begin{example}\label{S3}
    We consider the following object: $B=(S_3, s_1s_2s_2s_1s_1s_2, \pi_B=s_2s_1s_2,\{1,2,5\})$. The root function is \begin{align}
       & \rx_B(1)=\alpha_1, \rx_B(2)=\alpha_2, \rx_B(3)=\alpha_2, \nonumber \\
        & \rx_B(4)=\alpha_1+\alpha_2, \rx_B(5)=-(\alpha_1+\alpha_2),\rx_B(6)=\alpha_1. \nonumber
    \end{align}
    Take $A$ to be the subobject of $B$ given by the flat $\{1,6\}$. We have $A\sim (S_2, ss, s, \{1\})$. Take $C$ to be the subobject of $B$ given by the flat $\{2,3\}$ and $C'$ to be the subobject of $B$ given by the flat $\{4.5\}$. We have $C\sim (S_2, ss, s, \{1\})$ and $C'\sim (S_2, ss, s, \{2\})$. \par
    
    The sequences $A \overset{f}{\hookrightarrow} B \overset{h}{\twoheadrightarrow} C$ and $A \overset{f}{\hookrightarrow} B \overset{h'}{\twoheadrightarrow} C'$ are all admissible. In particular, $h\circ f=h'\circ f =0$. Here $f=([2],\{1,6\})$, $h=(\{2,3\},[2])$ and $h'=(\{4,5\},[2])$. However, since $C$ and $C'$ are not isomorphic, we can not find a morphism $l$ from $C$ to $C'$ such that $l\circ h=h'$. 
\end{example}
\begin{example}\label{S4}
    We consider the following object: \[B=(S_4, s_1s_2s_3s_1s_2s_2s_3, \pi_B=s_1s_2,\{1,2,3,5,7\}).\] The root function is \begin{align}
       & \rx_B(1)=\alpha_1, \rx_B(2)=\alpha_2, \rx_B(3)=\alpha_3, \nonumber \\
        & \rx_B(4)=\alpha_1, \rx_B(5)=\alpha_1+\alpha_2, \nonumber \\
        & \rx_B(6)=\alpha_1+\alpha_2, \rx_B(7)=\alpha_1+\alpha_2+\alpha_3 \nonumber
    \end{align}
    Take $A$ to be the subobject of $B$ given by the flat $\{1,3,4\}$, which is induced by the subspace spanned by $\{\alpha_1,\alpha_3\}$. We have $A\sim (S_2\times S_2, ss's, s, \{1,2\})$. Take $C$ to be the subobject of $B$ given by the flat $\{5,6\}$ and $C'$ to be the subobject of $B$ given by the flat $\{7\}$. We have $C\sim (S_2, ss, s, \{1\})$ and $C'\sim (S_2, s, e, \{1\})$. \par
    
    The sequences $A \overset{f}{\hookrightarrow} B \overset{h}{\twoheadrightarrow} C$ and $A \overset{f}{\hookrightarrow} B \overset{h'}{\twoheadrightarrow} C'$ are all admissible. In particular, $h\circ f=h'\circ f =0$. Here $f=([3],\{1,3,4\})$, $h=(\{5,6\},[2])$ and $h'=(\{7\},[1])$. However, since $C$ and $C'$ are not isomorphic, we can not find a morphism $l$ from $C$ to $C'$ such that $l\circ h=h'$. 
\end{example}


\begin{definition} \label{def:direct_sum_in_C}

Given $X,Y\in \CC$, the direct product of $W_X$ and $W_Y$ and the concatenation of the words $Q_X$ and $Q_Y$ defines a new element $X\oplus Y$ in $\mathscr{C}$, which is called the \textbf{direct sum} of $X$ and $Y$. 
\end{definition}

By definition, we have canonical embeddings $X \hookrightarrow X \oplus Y, Y \hookrightarrow X \oplus Y$ given by $([n_X], [n_X])$ and $([n_Y], n_X + [n_Y])$, respectively. The opposite morphisms give canonical retractions. 

\begin{remark}
\label{rem:direct_sum_not_(co)product}
Notice that $X\oplus Y$ is neither the product nor the coproduct of $X$ and $Y$ in the category $\mathscr{C}$. \par

For example, if we take $Y=X$ and denote the embedding from $X$ to the first (second) factor of $X\oplus X$ by $i_1$ (resp. $i_2$). There exists no morphism $h$ from $X\oplus X$ to $X$ such that the following diagram commutes:
\[
\xymatrix{
& X    & \\
X \ar[ur]^{id}  \ar[r]^{i_1} & X\oplus X \ar[u]^{h}  & X \ar[ul]_{id} \ar[l]_{i_2}
}
\]
Thus $X\oplus Y$ is not the coproduct of $X$ and $Y$ in $\mathscr{C}$ in general. A similar discussion also shows that $X\oplus Y$ is not the product of $X$ and $Y$ in $\mathscr{C}$ in general. 


\end{remark}
However, the direct sum $\oplus$ defines a symmetric monoidal structure on $\mathscr{C}$. Moreover, it satisfies the following properties.
\begin{proposition}\label{prop:dir_sum_universal_property}
    Given objects $X$, $Y$ and $Z$ in $\mathscr{C}$ with two morphisms $h:X\to Z$ and $g:Y\to Z$, there exists a unique morphism $h\oplus g$ from $X\oplus Y$ to $Z\oplus Z$ such that $p_1\circ (h\oplus g) \circ i_X = h$ and $p_2\circ (h\oplus g) \circ i_Y = g$. Here $p_1$ (resp. $p_2$) is the projection from the first (resp. second) factor of $Z\oplus Z$ to $Z$. 
\end{proposition}

\begin{proof}
    Suppose that $h=(H_1,H_2)$ and $g=(G_1,G_2)$. We define the morphism from $X\oplus Y$ to $Z\oplus Z$ by $h\oplus g:=(H_1\sqcup(G_1+n_X), H_2\sqcup (G_2+ n_Z))$. The following diagram commutes:
    \[
\xymatrix{
 X\ar[drr]^{i_X} \ar[ddddrr]_{h} & & & & Y \ar[dll]_{i_Y} \ar[ddddll]^{g} \\
&& X\oplus Y \ar[d]_{h\oplus g} && \\
&& Z\oplus Z \ar@<0.3ex>[dd]^{p_2} \ar@<-0.3ex>[dd]_{p_1}&& \\
\\
&& Z &&
}
\]\par
Moreover, if $h\oplus g=(F_1,F_2)$ is a morphism from $X\oplus Y$ to $Z\oplus Z$ such that the diagram above commutes, it has to be equal to  $(H_1\sqcup(G_1+n_X), H_2\sqcup (G_2+ n_Z))$. The reason is that the definition of the direct sum and the commutative condition force $F_1\cap[n_X]=H_1$ and $F_1\cap[n_X+1,n_X+n_Y]=G_1 +n_X$. 
\end{proof}

\begin{lemma} \label{lem:subquotient_of_direct_sum}
Given an admissible sequence $A \overset{f}\hookrightarrow B_1 \oplus B_2 \overset{g}\twoheadrightarrow C$,  there exist direct sum decompositions $A \overset{\overset{a}\sim}\to A_1 \oplus A_2, C \overset{\overset{c}\sim}\to C_1 \oplus C_2$ such that the maps $f, g$ factor as
\[
f = \left(\begin{matrix}f_1 & 0\\ 0 &f_2\end{matrix}\right)\circ a;
\,\,\,\,\,\,\,\,
g = \left(\begin{matrix}g_1 & 0\\ 0 &g_2\end{matrix}\right)\circ b.
\]
\end{lemma}

\begin{proof}
Suppose that $f=(F_1,F_2)$, so $A_{F_1}\sim (B_1\oplus B_2)_{F_2}$. We take a decomposition $F_2=(F_2\cap[n_{B_1}])\sqcup (F_2\cap[n_{B_1}+1,n_{B_2}+n_{B_1}])$. It induces a decomposition  $F_1=F_{1,1}\sqcup F_{1,2}$ such that $A_{F_{1,1}}\sim (B_1)_{F_2\cap[n_{B_1}]}$ and $A_{F_{1,2}}\sim (B_2)_{F_2\cap[n_{B_1}+1,n_{B_2}+n_{B_1}]}$. Since $f$ is a monomorphism, we have $F_1=[n_A]$. Take $A_1=F_{1,1}$ and $A_2=F_{1,2}$. Notice that $A_1$ is isomorphic to a subobject of $B_1$ and $A_2$ is isomorphic to a subobject of $B_2$. The fact that the Coxeter groups of $B_1$ and $B_2$ commute implies that the Coxeter groups of $A_1$ and $A_2$ commute. We have the decomposition of $A$ as desired. For the maps, we take $f_1=(F_{1,1}, F_2\cap[n_{B_1}])$ and $f_2=(F_{1,2}, F_2\cap[n_{B_1}+1,n_{B_2}+n_{B_1}])$.
 Same argument works for $C$ and $g$.
\end{proof}

\begin{definition}\label{def:cofinitary}
     A category $\mathcal{L}$ is \textbf{cofinitary} if for any $X\in \mathcal{L}$, there are only finitely many subobjects and quotient objects of $X$ (up to isomorphism).
 \end{definition}
 \begin{lemma} \label{lem:C_cofinitary}
     The category $\mathscr{C}$ is a cofinitary category. 
 \end{lemma}
 \begin{proof}
     By Proposition~\ref{prop:epi_mono_in_C}, two different subobjects of an object $X\in \CC$ are represented by $X_F$ and $X_{F'}$ for two different irreducible flats $F$ and $F'$ of $X$. Thus the subsets $F\cap I_X$ and $F'\cap I_X$ are different. Since $I_X$ only has finitely many subsets, the object $X$ has finitely many different subobjects. The same holds for quotient objects.
 \end{proof}

\begin{theorem}\label{thm:Hall_C}
    Fix a field $k$. The following defines a a connected graded commutative and cocommutative Hopf $k$-algebra $\mathcal{H}(\mathscr{C})=(H,m,u,\Delta,\epsilon)$.
    \begin{itemize}
        \item As a set, $H=\operatorname{span}_k(\operatorname{\pi}_0(\mathscr{C}))=\{\sum_{1\leq i\leq n}\lambda_i \cdot [X_i]\}_{n\in \mathbb{N}, \lambda_i \in k, X_i\in\CC}$, which is a $k$-vector space with a basis parameterized by
        the isomorphism classes of objects in $\mathscr{C}$;
        \item The multiplication of two 
        isomorphism
        classes $[X]\cdot [Y]$ is the isomorphism class of the direct sum of their arbitrary representatives $[X\oplus Y]$:
        \begin{align} \label{eq:mult_C}
            m([X],[Y]):=[X\oplus Y];
        \end{align}
        \item The unit $u$ sends $\lambda\in k$ to $\lambda\cdot[(W_{tri},\emptyset,\emptyset,\emptyset)]$;
        \item The comultiplication of an 
        isomorphism
        class $[X]$ is given by the following decompositions of an arbitrary representative of $X$:
        \begin{align} \label{eq:comult_C}
            \Delta ([X])= \sum_{\substack{A \overset{f}{\hookrightarrow} X \overset{g}{\twoheadrightarrow} C \\[4pt] \text{admissible,} \\ \text{$f$ canonical embedding,}
            \\\text{$g$ canonical retraction}
            }}
            [A]\otimes [C];
        \end{align}
        \item The counit $\epsilon$ sends $[X]$ to $1$ if $[X]=[(W_{tri},\emptyset,\emptyset,\emptyset)]$ and to $0$ otherwise. 
    \end{itemize}
    We call $\mathcal{H}(\mathscr{C})$ the \textbf{dual Hall Hopf algebra} of the category $\mathscr{C}$.
\end{theorem}

The sum in the definition of the comultiplication is finite for each $X \in \mathscr{C}$ thanks to Lemma~\ref{lem:C_cofinitary}. We note that an equivalent (and arguably more conceptual) way to define the comultiplication is by the formula
\begin{align} 
            \Delta ([X])= \sum_{A, C \in \text{Iso}(\mathscr{C})}
            g_{AC}^X [A]\otimes [C],
        \end{align}

where $A, C, X$ are representatives of their isomorphism classes, chosen arbitrarily, and $g_{AC}^X$ is the number of admissible exact sequences $A {\hookrightarrow} X {\twoheadrightarrow} C$. This is because for each pair of isomorphic objects $U, V \in \mathscr{C}$, there is a unique isomorphism $U \overset\sim\to V$. 

\begin{proof}
   By Proposition~\ref{prop:iso_in_C}, the 
   isomorphism
   classes of objects in $\mathscr{C}$ are precisely the $k$-basis elements in the Hopf algebra of subword complexes in \cite{BC}. As a Hopf algebra, the dual Hopf Hall algebra $\mathcal{H}(\mathscr{C})$ is just the Hopf algebra of subword complexes defined in \cite[Theorem 3.6]{BC}. Indeed, admissible sequences in \eqref{eq:comult_C} bijectively correspond to 2-flat-decompositions in \cite{BC}, direct sums in \eqref{eq:mult_C} are defined exactly as in \cite[Definition 3.5]{BC}, and the same goes for the unit and counit. The algebra $\mathcal{H}(\mathscr{C})$ is thus defined by the same operations as the Hopf algebra in \cite[Theorem 3.6]{BC}, as so is the same Hopf algebra, only defined
   using the category $\mathscr{C}$.
\end{proof}

Instead of this indirect proof, we could have proved the (co)associativity of the (co)multiplication directly, similarly to the usual proofs for Hall (co)algebras in abelian, exact, or proto-exact categories, see e.g. \cite{PRO, DK, Hu}. Then an almost trivial statement would have been the isomorphism of thus defined dual Hall algebra to the Hopf algebra of \cite{BC}.
We will discuss proto-exact categories and their Hall algebras in more detail in Section~\ref{sec:categories_of_subobjects_and_subquivers}.
In this approach, the coassociativity of the comultiplication comes from the fact that given $A, B, C, X$, there are as many diagrams 

\begin{equation} \label{eq:diag_assoc_1}
        \begin{tikzcd}
        A  \arrow[r,hook, "i_A"]  & Y \arrow[r,two heads,"p_B"] \arrow[d,hook, "i_Y"] & B \\
        & X \arrow[d,two heads, "p_C"] & \\
        & C &
        \end{tikzcd}
\end{equation}

as diagrams 

\begin{equation} \label{eq:diag_assoc_2}
        \begin{tikzcd}
        && B \arrow[d,hook,"i_B"]\\
        A  \arrow[r,hook,"i_A"]  & X \arrow[r,two heads,"p_Z"]  &  Z \arrow[d,two heads,"p_C"] \\
        && 
         C 
        \end{tikzcd}.
\end{equation}

This is the case because each of these diagrams can be completed to a commutative diagram of the form 

\begin{equation} \label{eq:diag_assoc_3}
        \begin{tikzcd}
        A  \arrow[r,hook,"i_A"] \arrow[-,double line with arrow={-,-}]{d} & Y \arrow[r,two heads,"p_B"] \arrow[d,hook,"p_B"] & B \arrow[d,hook,"i_B"]\\
       A  \arrow[r,hook,"i_A"] & X \arrow[r,two heads,"p_Z"] \arrow[d,two heads,"p_C"] & Z \arrow[d,two heads,"p_C"]\\
        & C \arrow[-,double line with arrow={-,-}]{r} & C 
        \end{tikzcd}, 
\end{equation}

and such a complement defines the other diagram by restriction. Further, given a diagram \eqref{eq:diag_assoc_1} or \eqref{eq:diag_assoc_2}, the complement to a diagram of the form \eqref{eq:diag_assoc_3} is unique (as long all the morphisms are indeed asked to be canonical embeddings and retractions).
 Indeed, for a diagram \eqref{eq:diag_assoc_1}, $A, B, C$ are all subobjects of the object $X$ given by its irreducible flats. Then $Z$ is the irreducible flat of $X$ defined as $\F(\V(I_X \cap (F_B \cup F_C)))$. Similarly, for a diagram \eqref{eq:diag_assoc_2}, $Y$ must be the flat $\F(\V(I_X \cap (F_A \cup F_B)))$ of $X$.

The multiplication defined in~\eqref{eq:mult_C} is clearly associative, and it is compatible with the comultiplication defined in~\eqref{eq:comult_C} thanks to Lemma~\ref{lem:subquotient_of_direct_sum}. %

\begin{remark}
The (co)commutativity of the (co)multiplication of $\mathcal{H}(\mathscr{C})$ can be deduced from the fact that we have a natural equivalence $D: \mathscr{C}^{op} \to \mathscr{C}$ which preserves and reflects admissible sequence and is such that $DX \cong X, \forall X \in \mathscr{C}$. See \cite[Remark 1.8]{PRO} for a similar remark on Hall algebras of proto-abelian categories. 
\end{remark}

We note that despite the coassociativity of its Hall comultiplication, the category $\mathscr{C}$ is not proto-exact. In proto-exact categories, in (the analogue of) an admissible sequence $A \overset{f}{\hookrightarrow} X \overset{g}{\twoheadrightarrow} C$, $f$ is the kernel of $g$ and $g$ is the cokernel of $f$. This is not the case for $\mathscr{C}$, as discussed in Remark~\ref{rem:not_weak_(co)kernels}. 

In fact, it seems also not to be the case that the structure of admissible sequences in $\mathscr{C}$ is induced by a proto-exact structure on some $(\infty,1)$-category with the homotopy category equivalent to $\mathscr{C}$. There, it appears from the definition in \cite{DK} that in  (the analogue of) an admissible sequence $A \overset{f}{\hookrightarrow} X \overset{g}{\twoheadrightarrow} C$, $f$ is a weak kernel of $g$ and $g$ is a weak cokernel of $f$. In the additive setting, the induced structure on the homotopy category is extriangulated, as defined by Nakaoka--Palu \cite{NakaokaPalu1,NakaokaPalu2}, and admissible sequences are indeed weak kernel-cokernel pairs, see e.g. \cite[Section 5]{BBGH} for a detailed discussion. To the best of our knowledge, in the non-additive case, similar induced structures have not been axiomatized, but one might expect the same property to hold. Again by Remark~\ref{rem:not_weak_(co)kernels}, this is not the case for $\mathscr{C}$. We also note that pulling back an exact or extriangulated structure along a nice enough functor does often define an exact (resp. extriangulated) structure on the source category, see \cite[Appendix A]{CKP} and references therein, but this uses restrictions on the class of admissible sequences not satisfied for $\mathscr{C}$, even though we can define them in terms of the functor $\mbox{For}$ and the split exact structure on $\Vect_{\mathbb{R}}$. 



\subsection{Categorical description of flips}

Now we describe the corresponding operation of flip on the categorical level.

\begin{definition}\label{def:flip}
    Given a subword complex $\SC_W(Q,\pi)$, we say that a maximal face $I'$ is obtained from  another maximal face $I$ by a \textbf{flip} along a subset $L$ of $I$ if there exists a subset $L'$ of $I'$ such that $I\backslash L=I' \backslash L'$. We denote the flip from $I$ to $I'$ by \begin{align}
      I \overset{\sigma_{L,L'}}{\longrightarrow} I'. \nonumber
    \end{align}  
    The \textbf{degree} of the flip $\sigma_{L,L'}$ is the cardinality of the set $I\backslash L$. A flip of degree $n$ is also called an $n$-flip.
\end{definition}

The incidence graph $G_W(Q,\pi)$ of a subword complex $\SC_W(Q,\pi)$ is a connected graph, whose vertices are maximal faces and two vertices are connected by an edge if and only if there is a 1-flip between them. \par

Notice that the category $\mathscr{C}$ is graded by the rank of the Coxeter group. The indecomposable objects are just objects of degree 1. In fact, 1-flips can be described using subobjects of degree 1 as follows:

\begin{definition}\label{def:flippable}
    Given $X\in \CC$, an element in $I_X$ is called a \textbf{folded position} and an element in $[n_X]\backslash I_X$ is called a \textbf{traversing position}. An irreducible flat $F$ is \textbf{flippable} if $\operatorname{deg}X_F=1$ and there exists a traversing position in $F$.
\end{definition}

This definition is motivated by the folded gallery model for objects in $\mathscr{C}$ in \cite{Zijun}.
Notice that any flippable irreducible flat $F$ of $X\in \CC$ has a unique traversing position since $[n_X]\backslash I_X$ gives a reduced expression of $\pi_X$. We denote the unique traversing position of $F$ by $t_F$. In other words, we have $F=\{t_F\}\cup (I_X\cap F)$.

\begin{proposition} \label{prop:flip_in_C}
    Fix $X\in\CC$ and a flippable irreducible flat $F$ of $X$. Any folded position $i\in I_X \cap F$ gives a 1-flip of $I_X$ along $\{i\}$ of the form \begin{align}
      I_X \overset{\sigma_{\{i\},\{t_F\}}}{\longrightarrow} I(i), \nonumber
    \end{align}  
    which gives $\theta_{F,i}(X)=(W_X, Q_X,\pi _X, I(i))\in \CC$.
    Conversely, any 1-flip of $I_X$ is of this form. 
\end{proposition}

\begin{proof}
   We use \cite[Lemma 2.6]{BC}. If $i<t_F$, we have $\rx_X(t_F)=\rx_X(i)$. And the root function of $Y=\theta_{F,i}(X)$ can be deduced from the root function of $X$ as follows: \begin{align}
        \rx_Y(j)=\begin{cases}
            s_{X,i}(\rx_X(j)) & \text{if } i<j\leq t_F \\
            \rx_X(j) & \text{else,}
        \end{cases} \nonumber
    \end{align}
     where $s_{X,i}$ denotes the reflection that is orthogonal to the root $\rx_X(i)$. In particular, we have $\rx_Y(i')=-\rx_X(i)$. The case $i>t_F$ is analogous. \par
\end{proof}

We consider the set $\tilde{S}_{n}$ containing all elements of the following form:
\begin{align}
    S_{n,i}=(W_1,\underbrace{s\cdots s}_{n \text{ terms}}, s, [n]\backslash \{i\}), \nonumber
\end{align}
where $W_1$ is the Coxeter group of rank 1 generated by $s$, and $1\leq i \leq n$. Given $X\in\CC$ and a flippable irreducible flat $F$ of $X$ such that $n_{X_F}=n$, $X_F$ is isomorphic to $S_{n,i}$ for some $1\leq i\leq n$ via a unique isomorphism. The symmetric group $S_n$ acts on $\{S_{n,i}\}_{1\leq i\leq n}$ naturally. We denote the action of the permutation $(i,j)$ on this set by $\tau_{i,j}$. \par
Using the notation of Proposition~\ref{prop:flip_in_C}, the 1-flip $\sigma_{\{i\},\{t_F\}}$ at the folded position $i$ corresponds to the following diagram:\par
\[
        \begin{tikzcd}
        X  \arrow[d] & \theta_{F,i}(X) \arrow[d]\\
        S_{n,s} \arrow[r, "\tau_{s,t}"] & S_{n,t} 
        \end{tikzcd}
        \]
Here $S_{n,s}$ is isomorphic to $X_F$ and $S_{n,t}$ is isomorphic to $\theta_{F,i}(X)_{F}$ for some $1\leq s,t\leq n$. Here $\tau_{s,t}$ denotes the action of the transposition $(s,t)$ in the symmetric group on the set $\tilde{S}_n$. All flips of $X$ can be described using a diagram above.

\section{A proto-abelian subcategory} \label{sec:root-independednt_subcat}

The category $\mathscr{C}$ has interesting subcategories. In this section, we construct a proto-abelian subcategory from root independent subword complexes. 

\begin{definition}\label{def:root independent}
 An object $X\in \CC$ is $\textbf{root independent}$ if the cardinality of $I_X$ is equal to the rank of the Coxeter group $W_X$. Equivalently, $X$ is root independent if and only if the root configuration $R(X)=\{\rx_X(i)\}_{i\in I_X}$ is a basis of $V_X$. 
\end{definition}

\begin{proposition}\label{prop:subobjects_D}
    Given a root independent object $X\in \CC$, any subset $J\subseteq I_X$ gives an irreducible flat $f(J)$ of $X$ such that $J=I\cap f(J)$. Moreover, $X_{f(J)}$ is also root independent. Conversely, any subobject of $X$ is given in this way.
\end{proposition}

\begin{proof}
    The roots in $\{\rx_X(i)\}_{i\in J}$ are linearly independent. Taking $V(J)$ the subspace of $V_X$ generated by $\{\rx_X(i)\}_{i\in J}$, we have $J=I_X\cap \operatorname{F}(V(J))$. Since $\{\rx_X(i)\}_{i\in J}$ is a basis of $V(J)$ and the root function is invariant by Lemma~\ref{lem:flats_quadruples_BC}, the object $X_{f(J)}$ is also root independent.\par
    Conversely, given any flat $F$ of a root independent object $X\in \CC$, we have $F=f(F\cap I_X)$ immediately.
\end{proof}

 Since the root function is invariant under taking subobjects, we have the following corollary:

\begin{corollary}\label{cor:subobject_D}
    Any subobject of a root independent object is still root independent.
\end{corollary}
Here we abuse the notation of the subobject slightly. By a subobject, we mean a chosen representative of the corresponding equivalence class. 
 Hence, the subcategory of root independent objects in $\CC$ is closed under taking subobjects. In fact, given an admissible sequence $A \overset{f}{\hookrightarrow} B \overset{h}{\twoheadrightarrow} C$ in $\mathscr{C}$, $A$ and $C$ are root independent if $B$ is root independent.

\begin{definition}\label{def:category_D}
    We denote the full subcategory of $\mathscr{C}$ of root independent objects by $\mathscr{D}$, which is called the \textbf{root independent subword complex category}. The restriction of the dual Hall Hopf algebra $\mathcal{H}(\mathscr{C})$ on $\mathscr{D}$ gives the dual Hall Hopf algebra $\mathcal{H}(\mathscr{D})$ of $\mathscr{D}$.
\end{definition}

Notice that in $\mathscr{C}$, an admissible sequence $A \overset{f}{\hookrightarrow} B \overset{h}{\twoheadrightarrow} C$ is not uniquely determined by $A$ and $B$, since a flat can have different complementaries. Thus we can not define quotients in $\mathscr{C}$. However, in the subcategory $\mathscr{D}$, we can define quotients. Moreover, we have the following complete description of admissible sequences in $\mathscr{D}$:

\begin{proposition} \label{prop:subquotients_in_D}
    The set of subobjects of  $X\in \mathscr{D}$ can be represented by $\{X_{f(J)}\}_{J\subseteq I_X}$. The sequence $X_{f(J)} \hookrightarrow X \twoheadrightarrow C$ is admissible if and only if $C$ is isomorphic to $X_{f(I\backslash J)}$.
\end{proposition}



\begin{proof}
    The first part is given by Corollary~\ref{cor:subobject_D}. Given an admissible sequence $X_{f(J)} \hookrightarrow X \twoheadrightarrow C$ in $\mathscr{D}$, the unique complementary of the flat $f(J)$ is $f(I\backslash J)$ since it must contain all elements in $I\backslash J$. Thus the sequence $X_{f(J)} \hookrightarrow X \twoheadrightarrow C$ is admissible if and only if $C$ is isomorphic to $X_{f(I\backslash J)}$.
\end{proof}

The concept of proto-abelian category is a generalization of abelian category, which allows us to define admissible extensions without the additive property. For such a category, we can generalize the definition of the Hall algebra for abelian categories. Proto-abelian categories also have interesting connections with higher Segal conditions. See \cite{DK}, \cite{PRO} for more details. There is an earlier slightly different notion of proto-abelian categories, see \cite[Remark 1.13]{Mozgovoy} and references therein. We use the notion in \cite[Definition 2.4.2]{DK}. 

\begin{definition}[Proto-abelian category]\label{def:proto-abelian category}
A category $\mathcal{E}$ is called \textbf{proto-abelian} if the following conditions hold:
\begin{enumerate}
    \item The category $\mathcal{E}$ is pointed.
        \item Every diagram in $\mathcal{E}$ of the form
        \[
        \begin{tikzcd}
        A  \arrow[r,hook] \arrow[d,two heads] & B \\
        C
        \end{tikzcd}
        \]
        can be completed to a pushout square of the form
        \[
        \begin{tikzcd}
        A  \arrow[r, hook] \arrow[d, two heads] & B \arrow[d, two heads]\\
        C \arrow[r, hook] & D 
        \end{tikzcd}
        \]
     \item Every diagram in $\mathcal{E}$ of the form
        \[
        \begin{tikzcd}
        & B \arrow[d, two heads] \\
        C \arrow[r, hook] & D
        \end{tikzcd}
        \]
        can be completed to a pullback square of the form
        \[
        \begin{tikzcd}
        A \arrow[r, hook] \arrow[d, two heads] & B \arrow[d, two heads] \\
        C \arrow[r, hook] & D
        \end{tikzcd}
        \]
    \item A commutative square in $\mathcal{E}$ of the form
   \[
        \begin{tikzcd}
        A \arrow[r, hook] \arrow[d, two heads] & B \arrow[d, two heads] \\
        C \arrow[r, hook] & D
        \end{tikzcd}
        \]
    is a pushout square if and only if it is a pullback square.
    We call such a square \textbf{biCartesian}.
\end{enumerate}
\end{definition}
\begin{remark}\label{remark:extension}
     In a proto-abelian category $\mathcal{E}$, a biCartesian diagram 
     \[
        \begin{tikzcd}
        A \arrow[r, hook] \arrow[d, two heads] & B \arrow[d, two heads] \\
        \mathbf{0} \arrow[r, hook] & C
        \end{tikzcd}
        \]
      is called an extension of $C$ by $A$, or a short exact sequence. \par
      
       Two extensions 
\[
A \hookrightarrow B \twoheadrightarrow A' 
\quad \text{and} \quad 
A \hookrightarrow C \twoheadrightarrow A'
\]
of $A'$ by $A$ are equivalent if there exists a commutative diagram
\[
\begin{tikzcd}
A \arrow[r, hook] \arrow[d, "id"] & B \arrow[r, two heads] \arrow[d, "\cong"] & A' \arrow[d, "id"] \\
A \arrow[r, hook, "id"'] & C \arrow[r, two heads] & A'
\end{tikzcd}
\]
We denote by $\operatorname{Ext}_{\mathcal{E}}(A',A)$ the set of equivalence classes of extensions of $A'$ by $A$. The middle terms of equivalent extensions are isomorphic by definition, and we further denote by  $\operatorname{Ext}_{\mathcal{E}}(A',A)_B$ the set of equivalence classes inside $\operatorname{Ext}_{\mathcal{E}}(A',A)$ with middle term isomorphic to $B$.
\end{remark}

We have the following nice property of $\mathscr{D}$:
\begin{theorem} \label{thm:proto-abelian_D}
    The category $\mathscr{D}$ is a proto-abelian category.
\end{theorem}

\begin{proof}
    The category $\mathscr{D}$ is pointed with $\mathbf{0}_{\mathscr{D}}=(W_{triv},\emptyset,\emptyset,\emptyset)$.\par
    We fix a diagram in $\mathscr{D}$ of the form
     \[
        \begin{tikzcd}
        A  \arrow[r,hook,"f"] \arrow[d,two heads,"g"] & B \\
        C
        \end{tikzcd}.
        \]
        
    We know that $A$ is a subobject of $B$ and $C$ is a subobject of $A$. Notice that any two isomorphic objects in $\mathscr{D}$ are isomorphic via a unique isomorphism. We can suppose that $A=B_{f(J)}$ and $C=B_{f(H)}$ for $H\subseteq J \subseteq I_B$ by Proposition~\ref{prop:subquotients_in_D}. For convenience, we identify $[n_A]$ and $[n_C]$ with subsets $f(J)$ and $f(H)$ of $[n_B]$ by the bijection $\operatorname{b}_{f(J)}$ and $\operatorname{b}_{f(H)}$ respectively. We claim that $D=B_{f((I_B\backslash J)\cup H)}$ gives a pushout square 
     \[
        \begin{tikzcd}
        A \arrow[r, hook, "f"] \arrow[d, two heads, "g"] & B \arrow[d, two heads, "k"] \\
        C \arrow[r, hook, "h"] & D
        \end{tikzcd}
        \]
        where $h$ is given by the inclusion $H\subseteq (I_B\backslash J)\cup H$ and $k$ is given by the inclusion $(I_B\backslash J)\cup H\subseteq I_B$. We identify $[n_D]$ with the subset $f(H\cup (I_B\backslash J))$ of $[n_B]$ via the bijection $\operatorname{b}_{f((I_B\backslash J)\cup H)}$. \par
        
    \textbf{Proof of the claim:}\par
    
    By the identification above, we have \begin{align}
        k\circ h=(f(H),f(H))=h\circ g. \nonumber
    \end{align}
    
    Suppose that we have a commutative diagram 
      \[    
        \begin{tikzcd}
        A \arrow[r, hook, "f"] \arrow[d, two heads, "g"] & B \arrow[d, two heads, "k"] \arrow[ddr, "k'"] \\
        C \arrow[drr, "h'"] \arrow[r, hook, "h"] & D \\
        & & X
        \end{tikzcd}
        \text{,}             
     \]
    where $k'=(R,R')\in \Mor(B,X)$ and $h'=(T,T')\in \Mor(C,X)$. The commutativity of the diagram implies the equation $k' \circ f=h' \circ g$. Notice that we have $I_B\cap f(H)=H$ and $I_B\cap f(J)=J$. By the definition of the composition of morphisms, we have  \begin{align}
        &(f(J)\cap R,\bar{R}')=(R,R')\circ (f(J),f(J))\nonumber \\ 
        &=k'\circ f\nonumber \\
        &=h'\circ g\nonumber \\
        &=(T,T')\circ (f(H),f(H))\nonumber \\
        &=(T,T'), \label{eq:composition in proto-abelian}
    \end{align}   
    where $\bar{R}'$ is the flat of $X$ contained in $R'$ induced by $J\cap R$. \par
    
    We have $R\cap f(J)=T\cap f(H)$. Thus we have $R\cap I_B\subset H\cup (I_B\backslash J)$, which implies the inclusion $R\in f(H\cup (I_B\backslash J))$. Thus we have $\psi=(R,R')\in \Mor(D,X)$. We have the following diagram:  \[
        \begin{tikzcd}
        A \arrow[r, hook, "f"] \arrow[d, two heads, "g"] & B \arrow[d, two heads, "k"] \arrow[ddr, "k'"] \\
        C \arrow[drr, "h'"] \arrow[r, hook, "h"] & D \arrow[dr, dashed, "\psi"] \\
        & & X
        \end{tikzcd}
        \text{.}
        \] 
        We have $\psi \circ k=(R,R')=k'$ by the inclusion $R\subset f(H\cup (I_B\backslash J))$. Moreover, we have $\psi \circ h=(R\cap f(J),R')=(T, T')=h'$. The first equality is given by the equation $R\cap f(J)=R\cap f(H)$. The second equality is given by Equation~\eqref{eq:composition in proto-abelian}. Thus the morphism $\psi$ makes the diagram above commutes.\par
        The morphism $\psi$ is unique. For any morphism $\tilde{\psi}=(\tilde{R},\tilde{R'})\in \Mor(D,X)$ such that $\tilde{\psi}\circ k=k'$, we have $\tilde{R}=\tilde{R}\cap f(H\cup (I_B\backslash J))=R$. The inclusion $\tilde{R}\subset f(H\cup (I_B\backslash J))$ implies $\tilde{\psi}\circ k=(\tilde{R},\tilde{R'})=(R,R')$. Thus we have $\tilde{R'}=R'$ and $\tilde{\psi}=\psi$. As a consequence, $(D,\psi)$ is the pushout that we desire. \par
        By the symmetry between the morphism sets $\Mor(X,Y)$ and $\Mor(Y,X)$ for any two objects $X,Y\in \operatorname{Ob}(\mathscr{D})$, we can obtain the existence of the pullback square by using the same construction for the opposite category $\mathscr{D}^{op}$. Moreover, if the square 
        \[
        \begin{tikzcd}
        A \arrow[r, hook] \arrow[d, two heads] & B \arrow[d, two heads] \\
        C \arrow[r, hook] & D
        \end{tikzcd}
        \]
        is a pullback square and we identify $D$ with $B_{f(J')}$ and $C$ with $B_{f(H)}$ for subset $H\subset J' \subset I_B$, the pullback $A$ can be given given by
        \begin{align}
            A=B_{f(H\cup (I_B\backslash J'))}. \label{eq:A in D}
        \end{align}
        For the biCartesian square, we consider the following commutative square in $\mathscr{D}$:
        \[
        \begin{tikzcd}
        A \arrow[r, hook] \arrow[d, two heads] & B \arrow[d, two heads] \\
        C \arrow[r, hook] & D
        \end{tikzcd}
        \]
        Suppose that this is a pushout square, by the discussion above, we can identify $A$ with $B_{f(J)}$ and $C$ with $B_{f(H)}$ for subset $H\subset J \subset I_B$. By the discussion above, we know that the pushout $D$ is given by $B_{f(H\cup (I_B\backslash J))}$. Using Equation~\eqref{eq:A in D} and the equality \begin{align}
            J=H\cup (I_B\backslash (I_B\backslash J)), \nonumber
        \end{align}
        the square is also a pullback square. The same argument for $\mathscr{D}^{op}$ implies that $D$ is a pushout if $A$ is a pullback. \par
        In conclusion, the category $\mathscr{D}$ is a proto-abelian category.
\end{proof}

\begin{remark}\label{remark: exact}
    In $\mathscr{D}$, a sequence \[ \begin{tikzcd}
        A \arrow[r, hook, "i"] & B \arrow[r, two heads, "j"] & C 
        \end{tikzcd}
        \] is proto-exact if and only if $i$ is a kernel of $j$ and $j$ is a cokernel of $i$.
\end{remark}
\begin{definition}\label{def:finitary proto-exact}
    A proto-abelian $\mathcal{L}$ is \textbf{finitary} if the morphism set $\operatorname{Mor}(X,Y)$ and the extension set $\operatorname{Ext}_{\mathcal{L}}(X,Y)$ are finitary for any two objects $X,Y\in \operatorname{Ob}(\mathcal{L})$. 
\end{definition}

By Proposition~\ref{prop:morphisms_in_C_finite}, the root independent subword complex category $\mathscr{D}$ has finite morphism sets. However, the extension set $\operatorname{Ext}_{\mathscr{D}}(X,Y)$ is not finite in general. To get a finitary proto-abelian category, we need a finiteness condition.

\begin{definition}
    The full subcategory of $\mathscr{D}$ consisting of objects 
    $X = (W_X, Q_X, \pi_X, I_X)$ with $W_X$ finite
    is denoted by $\mathscr{D}_{fin}$.
\end{definition}

\begin{theorem}
    The category $\mathscr{D}_{fin}$ is a finitary proto-abelian category.
\end{theorem}

\begin{proof}
   Any parabolic subgroup of a finite Coxeter group is still a finite Coxeter group. Thus the category $\mathscr{D}_{fin}$ is proto-abelian. We only need to prove that the extension set is finite.  Given $X,Y\in \operatorname{Ob}(\mathscr{D}_{fin})$, we want to prove the finiteness of the extension set $\operatorname{Ext}_{\mathscr{D}_{fin}}(X,Y)$. Given an admissible sequence \begin{align}
       X \overset{f}{\hookrightarrow} Z \overset{h}{\twoheadrightarrow} Y ,\nonumber
    \end{align}
    the maximal faces $I_X$ and $I_Y$ of $X$ and $Y$ gives a partition of $I_Z$. The rank $n=\operatorname{deg}(Z)$ of the finite Coxeter group $W_Z$ is equal to the sum of he rank of the finite Coxeter groups $W_X$ and $W_Y$. We only have finitely many isomorphism classes of finite Coxeter groups of rank $n$. \par
    For a fixed Coxeter group $W_Z$, the length of the word $Q_x$ is equal to the sum of $n$ and the length of a reduced expression of the group element $\pi_Z$, which is bounded by $n+l(W_Z)$, where $l(W_Z)$ is the length of the longest element in $W_Z$. In this case, $Q_Z$ has $n^{n+l(W_Z)}$ possible choices. \par
    Fixed $W_Z$ and $Q_Z$, the possible choices of $I_Z$ and $\pi_Z$ is finite since $W_Z$ is a finite group. \par
    In conclusion, the extension set $\operatorname{Ext}_{\mathscr{D}_{fin}}(X,Y)$ is finite.
\end{proof}

Using the 2-Segal condition of $\mathscr{D}_{fin}$, we can define its functorial Hall algebra by \cite[Proposition 2.19]{PRO}. It is is dual to a sub-Hopf algebra of $\mathcal{H}(\mathcal{\mathscr{C}})$. 

\section{root configuration quiver and proto-exact structure on the category \texorpdfstring{$\mathcal{S}_X$}{SX}}
\label{sec:categories_of_subobjects_and_subquivers}

We define the subquiver category $\mathcal{S}_X$ for any object $X$ in $\mathscr{D}$. We associate a root configuration quiver $\Gamma_X$ to $X$. When $\Gamma_X$ is of tree type, we give a proto-exact
structure on $\mathcal{S}_X$ using the partial order induced by the root configuration quiver of $X$. Moreover, we show that the Hall algebra of $\mathcal{S}_X$ induced by this proto-exact structure is isomorphic to the Hall algebra of the proto-abelian category $Rep(\Gamma_X,\mathbb{F}_1)$. In particular, we establish a categorification of the universal enveloping algebra $U(\mathfrak{n}_+)$ in type $A$. In this section, for any subobject $X_F$ of $X$, we label letters of $Q_{X_F}$ using $F$ instead of $[|F|]$ to simplify the notations. \par


\subsection{Root configuration quiver for objects in \texorpdfstring{$\mathscr{D}$}{D}}

We recall some basic properties about quiver representations over $\mathbb{F}_1$.
\begin{definition}\label{def:quiver}
Let $Q=(Q_0,Q_1,s,t)$ be a quiver, where $Q_0$ is the set of vertices,
$Q_1$ is the set of arrows, and $s,t:Q_1\to Q_0$ are the source and target maps.
The \textbf{path category} $\mathbf C(Q)$ associated to $Q$
is defined as follows:
\begin{itemize}
  \item The objects of $\mathbf C(Q)$ are the vertices of $Q$, i.e.\ $\operatorname{Ob}(\mathbf C(Q))=Q_0$.
  \item For $i,j\in Q_0$, the morphisms from $i$ to $j$ are given by
  all finite oriented paths in $Q$ from $i$ to $j$.
  In particular, for each vertex $i\in Q_0$, there is a length-zero path
  $\id_i:i\to i$, which serves as the identity morphism.
  \item Composition of morphisms is given by concatenation of paths.
\end{itemize}
\end{definition}

\cite[Theorem 5]{quiverf1} gives a description of indecomposable modules of a tree quiver over $\mathbb{F}_1$, which is used to construct a surjective algebra homomorphism from the universal enveloping algebra to the Hall algebra of $Rep(Q,\mathbb{F}_1)$ using the Lie algebra morphism. We show that in this case, our construction gives a similar result. \par

\begin{definition}\label{def:F_1}
 Let $Vect(\mathbb{F}_1)$ denote the category of finite vector spaces over $\mathbb{F}_1$. Objects of $Vect(\mathbb{F}_1)$ are pointed finite sets. A pointed set is a set $V$ with an element $0_V\in V$. The morphism set is defined as $\operatorname{Hom}(V,W)=\{\text{maps } f: V\to W \mid f(0_V)=0_W, f_{\mid V\backslash f^{-1}(0_W)} \text{is injective}\}$. A representation of $Q$ over $\mathbb{F}_1$ is defined as a functor from $\mathbf{C}_Q$ to $Vect(\mathbb{F}_1)$.
\end{definition}

\begin{remark}\label{remark:ks}
    The Krull-Schmidt theorem \cite[Theorem 4]{quiverf1} holds for quiver representations over $\mathbb{F}_1$. The following theorem implies that for a quiver $Q$ of tree type, $Q$ only has finitely many indecomposable representations over $\mathbb{F}_1$. 
    \end{remark}

\begin{theorem}\cite[Theorem 5]{quiverf1} \label{thm:indec_reps}
    If the underlying graph of a quiver $Q$ is a tree, indecomposable representations of $Q$ over $\mathbb{F}_1$ correspond to connected subquivers of $Q$ bijectively. We call such a quiver a \textbf{quiver of tree type}. 
\end{theorem}

In fact, \cite[Theorem 5.3]{JunSistko23} shows that a connected quiver is of tree type and only if it has finitely many isomorphism classes of indecomposable representations over $\mathbb{F}_1$.

For any quiver $Q$, we denote the Hall algebra of the category of nilponent representations of $Q$ over $\mathbb{F}_1$ defined in \cite{quiverf1} by $\mathbf{H}(Q):=\mathbf{H}(Rep(Q,\mathbb{F}_1))_{nil}$.

\begin{theorem}\cite[Theorem 8]{quiverf1} \label{thm:univ_enveloping}
   For any quiver $Q$ without loops, we denote the corresponding Kac-Moody algebra by $U(\mathfrak{g})$ with the canonical triangular decomposition \begin{align*}
       U(\mathfrak{g})\simeq U(\mathfrak{n}_{-})\otimes U(\mathfrak{h})\otimes U(\mathfrak{n}_{+}).
   \end{align*}
   We have a Hopf algebra homomorphism $\rho_Q$ from $U(\mathfrak{n}_{+})$ to $\mathbf{H}(Q)$. In particular, the homomorphism $\rho_Q$ is surjective if $Q$ is a tree quiver. When $Q$ is of type $A_n$, the homomorphism $\rho_Q$ is an isomorphism.
\end{theorem}

\begin{remark}
    The image of $\rho_Q$ is called the \textbf{composition subalgebra} of $\mathbf{H}(Q)$.
\end{remark}

Now we associate with any object in $\mathscr{D}$ a quiver using the root function.

\begin{definition}\label{def:root configuration quiver}
 For any $X=(W,Q,\pi,I)\in \mathscr{D}$, we denote the length of $Q$ by $n$ and the cardinality of $I$ by $m$. The root function of $X$ is denoted by $\rx$. The \textbf{root configuration quiver} $\Gamma_X$ of $X$ is a labeled quiver without multiple arrows defined as follows:
 \begin{enumerate}
     \item The vertex set of $\Gamma_X$ is labeled by $I$.
     \item There exists an arrow from $i$ to $j$ in $\Gamma_X$ if and only if \begin{align} \label{eq:arrows}
         (i-j)\cdot\langle \rx (i), \rx(j)\rangle >0.
     \end{align}
 \end{enumerate}
\end{definition}
In Equation~\eqref{eq:arrows}, the form $\langle\cdot , \cdot \rangle$ is the canonical symmetric bilinear form on $V\times V$ induced by the Coxeter matrix of $W$. Precisely, $\langle\cdot , \cdot \rangle$ is defined by the condition $\langle\alpha_i,\alpha_j\rangle=-cos(\frac{\pi}{\operatorname{ord}(s_is_j)})$.\par

\subsection{Subquiver category}\label{subsection: subquiver category}


For any $X=(W,Q,\pi,I)\in \mathscr{D}$, the set of subobjects of $X$ in $\mathscr{D}$ is represented by $\operatorname{Sub}(X)=\{X_F\}_{F \text{ is an irreducible flat of } X}$, which is a finite set. By slight abuse of notation, we identify $X_F$ with the 
isomorphism
class of subobjects of $X$ it represents, whenever this causes no ambiguity. Recall that any subset $J$ of $[n_X]$ induces a flat $f(J)=\F(\V(J))$ of $X$. In particular, any subset $J$ of $I$ induces a subobject $X_{f(J)}$ of $X$, which is denoted by $X(J)$. Conversely, any irreducible flat of $X$ is of the form $f(J)$ for a subset $J\subseteq I$. We have $\operatorname{Sub}(X)=\{X(J)\}_{J\subseteq I}$. \par

To simplify notations, we consider the skeleton category $\operatorname{Sk}(X)$ of the non-full subcategory of subobjects of $X$, whose object set is $\operatorname{Sub}(X)$ and morphisms are canonical embeddings defined in Definition~\ref{def:subquotients_in_C}. We consider the span category of the slice category of $\operatorname{Sk}(X)$ over $X$. 

Precisely, we have the following definition:

\begin{definition}\label{def:span_X}
    The category $\spanx(\operatorname{Sk}(X)/X)$ is defined as follows:
    \begin{enumerate}
        \item Every object in $\spanx(\operatorname{Sk}(X)/X)$ is a diagram of the form 
        \[
        \begin{tikzcd}
       a_J: X(J)  \arrow[r,hook, "i_J"]  & X ,
        \end{tikzcd}
        \]
        where $i_J=(f(J),f(J))$ is the canonical embedding. (Recall that we label letters of $Q_{X(J)}$ using $f(J)$ instead of $[|f(J)|]$.) Here $J$ ranges over all subsets of $I$.
        \item For $J,H \subseteq I$, the morphism set $\Mor(a_J,a_H)$ between the object \[
        \begin{tikzcd}
       a_J: X(J)  \arrow[r,hook, "i_J"]  & X ,
        \end{tikzcd}
        \] and 
        \[
        \begin{tikzcd}
        a_H: X(H)  \arrow[r,hook, "i_H"]  & X ,
        \end{tikzcd}
        \]
        is the set of diagrams \[
        \begin{tikzcd}
       a_K: X(K)  \arrow[r,hook, "i_K"]  & X ,
        \end{tikzcd}
        \]
        where $K$ ranges over subsets of $J\cap H$. In other words, $X(K)$ ranges over the set of subobjects of $X$ such that the following diagram commutes:
       
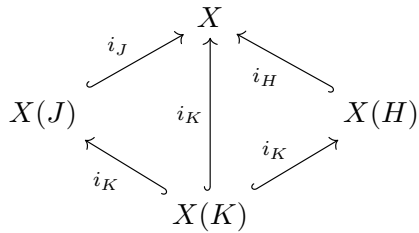
\[
 \begin{tikzcd}
        &  X  & \\
        X(J)  \arrow[ur,hook, "i_J"]  &  & X(H) \arrow[ul,hook, "i_H"] \\
        & X(K) \arrow[ur,hook, "i_K"] \arrow[uu,hook, "i_K"] \arrow[ul,hook, "i_K"]
        \end{tikzcd}
       \]
       \item For $J,H,L \subseteq I$, the composition of $a_K\in \Mor(a_J,a_h)$ and $a_M\in \Mor(a_H,a_L)$ is the diagram $a_{N}$, where $N=K\cap M$. In other words, the following diagram commutes: 
       
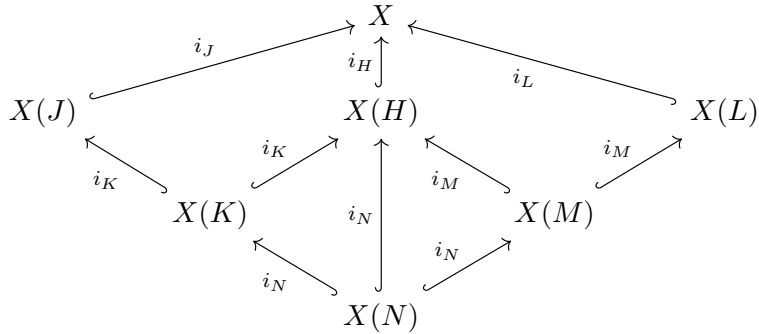
\[
 \begin{tikzcd}
       & &  X  & & \\
        X(J)  \arrow[urr,hook, "i_J"]  & & X(H) \arrow[u,hook, "i_H"] & & X(L) \arrow[ull,hook, "i_L"] \\
        & X(K) \arrow[ur,hook, "i_K"] \arrow[ul,hook, "i_K"] & & X(M) \arrow[ur,hook, "i_M"] \arrow[ul,hook, "i_M"] & \\
        & & X(N) \arrow[uu,hook, "i_N"] \arrow[ur,hook, "i_N"] \arrow[ul,hook, "i_N"] & &
        \end{tikzcd}
       \]
    \end{enumerate}
\end{definition}

Since the intersection of sets is associative, the composition of morphisms in $\spanx(\operatorname{Sk}(X)/X)$ is associative. In fact, it is isomorphic to the following category of subquivers on $\Gamma_X$: 
\begin{definition}\label{def:L_X}
    For any $X\in \mathscr{D}$, we consider the following category $\mathcal{L}_X$:\begin{itemize}
        \item Objects in $\mathcal{L}_X$ are labeled subquivers of $\Gamma_X$.
        \item For two subquivers $Q,Q'$ of $\Gamma_X$, morphisms from $Q$ to $Q'$ are given by label preserving partial isomorphisms of quivers. Recall that a label preserving partial isomorphism is a diagram \begin{align}
        Q \hookleftarrow P \hookrightarrow Q' \nonumber
    \end{align} given by an $I$-labeled quiver $P$, which is a subquiver of both $Q$ and $Q'$. 
    \item Compositions of label preserving partial isomorphisms 
        $Q \hookleftarrow P \hookrightarrow Q'$ and $Q' \hookleftarrow P' \hookrightarrow Q''$ is given by the intersection of $P$ and $P'$ in $Q'$.
    \end{itemize} 
\end{definition}
\begin{proposition}\label{prop:isomorphism_categories_spans}
    The category $\spanx(\operatorname{Sk}(X)/X)$ is isomorphic to the category $\mathcal{L}_X$.
\end{proposition}
\begin{proof}
    For any subset $H\subseteq I$, we send $a_H\in \spanx(\operatorname{Sk}(X)/X)$ to the full subquiver with vertex set $H$ of $\Gamma_X$. It is easy to check that this gives an isomorphism of categories.
\end{proof}
In fact, the category $\mathcal{L}_X$ can be viewed as a subcategory of the following subquiver category $\mathcal{S}_X$.
\begin{definition} \label{def:subquiver_category}
    For any $X\in \mathscr{D}$, we define the \textbf{subquiver category} $\mathcal{S}_X$ to be the category whose objects are finite disjoint unions of labeled quivers $\bigsqcup^t_{i=1} \Gamma_{X_i}$, where $X_i=X_{F_i}$ for a certain irreducible flat $F_i$ of $X$. Morphisms are given by label preserving partial isomorphisms of quivers. A subobject $X_F$ of $X$ represented by an irreducible flat $F$ is \textbf{indecomposable} if the quiver $\Gamma_{X_F}$ associated to $X_{F}$ is connected. By abuse of notation, we also call the object $\Gamma_{X_F}$ an indecomposable object in $\mathcal{S}_X$.
\end{definition}

Notice that any finite direct sum $\oplus_{i=1}^tX_i$ of indecomposable subobjects $X_i$ of $X$ gives an element $\bigsqcup^t_{i=1} \Gamma_{X_i}$ in $\mathcal{S}_X$. Conversely, any connected quiver in $\mathcal{S}_X$ is given by an indecomposable subobject of $X$. And any object in $\mathcal{S}_X$ is a finite disjoint union of connected quivers.

\begin{remark}\label{remark: disjoint union}
    The disjoint union of quivers gives a symmetric monoidal structure on $\mathcal{S}_X$. We call $\Gamma_1 \sqcup \Gamma_2$ the \textbf{direct sum} of $\Gamma_1$ and $\Gamma_2$ in $\mathcal{S}_X$. The direct sum here is neither the product or the coproduct. It satisfies the universal property given in Proposition~\ref{prop:dir_sum_universal_property}. Direct sums of morphisms are defined component-wise.
\end{remark}

\begin{proposition} \label{prop:subquivers_subobjects}
    For any subobject $X_F$ of $X$ given by an irreducible flat $F$,  $\Gamma_{X_F}$ is a subquiver of $\Gamma_X$. Conversely, any subquiver of $\Gamma_X$ is of the form $\Gamma_{X_F}$ for an irreducible flat $F$.
\end{proposition}

\begin{proof}
    Lemma~\ref{lem:flats_quadruples_BC} implies that the root function is invariant under restriction to irreducible flats. Thus Equation~\eqref{eq:arrows} is invariant under restriction to irreducible flats. \par
    Conversely, for any connected subquiver $Q$ of $\Gamma_X$, the vertex set $Q_0$ of $Q$ is a subset of $I_X$. We have $Q=\Gamma_{X(Q_0)}$ since $X$ is root independent. 
\end{proof}

\begin{proposition}\label{prop:L_X in S_X}
    $\mathcal{L}_X$ is isomorphic to a subcategory of $\mathcal{S}_X$.
\end{proposition}
\begin{proof}
    Since $X\in \mathscr{D}$, by Proposition~\ref{prop:subquotients_in_D}, any subset $H\subseteq I$ determines a unique irreducible flat $f(H)$ of $X$ . By Proposition~\ref{prop:subquivers_subobjects}, $\mathcal{L}_X$ is just isomorphic to the full subcategory of $\mathcal{S}_X$ whose objects are given by subsets of $I$.
\end{proof}

\begin{theorem} \label{thm:bij_indec_reps_to_subobjects}
    For any $X\in \mathscr{D}$ such that the underlying graph of $\Gamma_X$ is a tree, there is a bijection $\Psi_X$ from the set of indecomposable representations of $\Gamma_X$ over $\mathbb{F}_1$ to the set of indecomposable subobjects of $X$. 
\end{theorem}

\begin{proof}
    By Theorem~\ref{thm:indec_reps}, indecomposable representations of $\Gamma_X$ correspond to connected subquivers of $\Gamma_X$. Precisely, for any subquiver $Q$ of $\Gamma_X$, we set a one-dimensional $\mathbb{F}_1$-vector space at any vertex of $Q$ and $\mathbb{F}_1$-linear maps between them are identity maps. The vector spaces associated to other vertices of $\Gamma_X$ are zero-dimensional spaces. Proposition~\ref{prop:isomorphism_categories_spans} implies that any such connected subquiver $Q$ corresponds to a unique indecomposable subobject $X(Q_0)$. Here $Q_0$ is the vertex set of $Q$.
\end{proof}

\subsection{Proto-exact structure on \texorpdfstring{$\mathcal{S}_X$}{SX}}

We use the root configuration quiver to define a proto-exact structure on $\mathcal{S}_X$.

\begin{definition}[Proto-exact category {\cite[Definition 2.4.2]{DK}}]\label{def:proto-exact category}
A \textbf{proto-exact category} is a category $\mathcal E$ equipped with two classes
of morphisms $\mathcal M, \mathcal P$, whose elements are called
\textbf{admissible monomorphisms} and \textbf{admissible epimorphisms} such that the
following conditions are satisfied:
\begin{enumerate}
  \item[(PE1)] $\mathcal E$ is pointed, i.e.\ has an object $0$ which is both
  initial and final. Any morphism $0 \to A$ is in $\mathcal M$, and any morphism
  $A \to 0$ is in $\mathcal E$.

  \item[(PE2)] The classes $\mathcal M, \mathcal P$ are closed under composition
  and contain all isomorphisms.

  \item[(PE3)] A commutative square \[
        \begin{tikzcd}
        A \arrow[r, hook,"i"] \arrow[d, two heads, "j"] & B \arrow[d, two heads, "j'"] \\
        C \arrow[r, hook, "i'"] & D
        \end{tikzcd}
        \]
        in $\mathcal E$ with
  $i,i'$ admissible monomorphisms and $j,j'$ admissible epimorphisms
  is Cartesian if and only if it is coCartesian.

  \item[(PE4)] Any diagram in $\mathcal E$
  \[
    B \xrightarrow{j'} D \xleftarrow{i'} C
  \]
  with $i'$ an admissible monomorphism and $j'$ an admissible epimorphism,
  can be completed to a biCartesian square in (PE3) with
  $i$ an admissible monomorphism and $j$ an admissible epimorphism.

  \item[(PE5)] Any diagram in $\mathcal E$
  \[
    C \xleftarrow{j} A \xrightarrow{i} B
  \]
  with $i$ an admissible monomorphism and $j$ an admissible epimorphism,
  can be completed to a biCartesian square in (PE3) with
  $i'$ an admissible monomorphism and $j'$ an admissible epimorphism.
\end{enumerate}
\end{definition}

Given an object $X=(W,Q,\pi,I) \in \mathscr{D}$ with root configuration quiver $\Gamma_X$. We suppose that  $\Gamma_X$ is of tree type. Then we have a poset structure on $I$:
\begin{definition}[Poset structure on $I$]\label{def:poset}
   For $a,b\in I$, we set $a\preceq b$ if and only if there exists a path from $b$ to $a$ in $\Gamma_X.$ 
\end{definition}


\begin{definition}[Admissible monomorphisms and epimorphisms] \label{def:admissible_in_S_X}
    For any subsets $J\subseteq H\subseteq I$, $\Gamma_{X(J)}$ is a subquiver of $\Gamma_{X(H)}$. The morphism \begin{align}
       \tau_{J,H}:= \Gamma_{X(J)} \hookleftarrow \Gamma_{X(J)} \hookrightarrow\Gamma_{X(H)} \nonumber
    \end{align} is a \textbf{basic admissible monomorphism} if and only if  for any $a\in J$, the lower order ideal $\{b\in H \mid b\preceq a\}$ is contained in $J$.\par
    Symmetrically, for any subsets $K\subseteq H \subseteq I$, the  morphism \begin{align}
       \beta_{H,K}:= \Gamma_{X(H)} \hookleftarrow \Gamma_{X(K)} \hookrightarrow\Gamma_{X(K)} \nonumber
    \end{align}is a \textbf{basic admissible epimorphism} if and only if for any $a\in K$, the upper order ideal $\{b\in K \mid a\preceq b\}$ is contained in $K$. \par
    The class $\mathcal{M}$ of admissible monomorphisms consists of morphisms of the form \begin{align}
       \bigsqcup_{i=1}^m( \Gamma_{X(J_i)} \hookleftarrow \Gamma_{X(J_i)} \hookrightarrow\Gamma_{X(H_i)} )\nonumber,
    \end{align} where $m\in \mathbb{N}$ and each term \begin{align}
        \Gamma_{X(J_i)} \hookleftarrow \Gamma_{X(J_i)} \hookrightarrow\Gamma_{X(H_i)} \nonumber
    \end{align} is a basic admissible monomorphism for all $1\leq i\leq m$. The class $\mathcal{P}$ of admissible epimorphisms consists of morphisms of the form \begin{align}
       \bigsqcup_{i=1}^m( \Gamma_{X(H_i)} \hookleftarrow \Gamma_{X(K_i)} \hookrightarrow\Gamma_{X(K_i)} )\nonumber,
    \end{align} where $m\in \mathbb{N}$ and each term \begin{align}
        \Gamma_{X(H_i)} \hookleftarrow \Gamma_{X(K_i)} \hookrightarrow\Gamma_{X(K_i)} \nonumber
    \end{align} is a basic admissible epimorphism for all $1\leq i\leq m$. 
\end{definition}

Here we use notations in Section~\ref{subsection: subquiver category}. Any subset $J\subseteq I$ gives a subobject $X(J)$ of $X$. Now we give a proto-exact structure on $\mathcal{S}_X$.

\begin{theorem} \label{thm:proto-exact_S_X}
    When $\Gamma_X$ is a tree, the pair $(\mathcal{M},\mathcal{P})$ defines a proto-exact structure on $\mathcal{S}_X$.
\end{theorem}

\begin{proof}
    The zero object $0$ in $\mathcal{S}_X$ is given by the empty set $\emptyset\subseteq I$. Any morphism starting from (resp. ending at) $0$ is admissible by definition. \par
    The class $\mathcal{M}$ and $\mathcal{P}$ are closed under composition since lower (or upper) order ideals are closed under intersection. Any isomorphism of the form $\tau_{H,H}=\beta_{H,H}$ is both a monomorphism and a epimorphism since for any $a\in H$, the lower order ideal $\{b\in H \mid b \preceq a\}$ and the upper bound ideal $\{b\in H \mid a \preceq b\}$ are both subsets of $H$. By Definition~\ref{def:subquiver_category}, a morphism $\psi$ in $\mathcal{S}_X$ is an isomorphism if and only if $\psi=\sqcup_{i=1}^m\psi_i$ and $\psi_i=\tau_{H_i,H_i}=\beta_{H_i,H_i}$ for some subset $H_i\subseteq I$. By Definition~\ref{def:admissible_in_S_X}, $\psi$ lies in both $\mathcal{M}$ and $\mathcal{P}$ since all $\psi_i$ do. Thus all isomorphisms are contained in both $\mathcal{M}$ and $\mathcal{P}$.  \par
    To prove (PE3)-(PE5), without loss of generality, we take $B=\Gamma_{X(K)}$ for some $K\subseteq I$. Then we generalize to $\mathcal{S}_X$ using the direct sum $\sqcup$ since morphisms in $\mathcal{S}_X$ are defined component-wise. Precisely, a  morphism of the form $\sqcup_{i=1}^m\psi_i$ is in $\mathcal{M}$ (respectively, in $\mathcal{P}$) if and only if $\psi_i$ is in $\mathcal{M}$ (respectively, in $\mathcal{P}$) for all $1\leq i\leq m$. \par
    Suppose that we have a diagram \[
    B \xrightarrow{j'} D \xleftarrow{i'} C
  \]
  where $j'\in \mathcal{P}$ and $i'\in \mathcal{M}$. Then there exist subsets $J\subseteq H\subseteq K\subseteq I$ such that $B=\Gamma_{X(K)}$, $D=\Gamma_{X(H)}$ and $C=\Gamma_{X(J)}$. We take $A=\Gamma_{X(H')}$, where $H'=(K\backslash H)\cup J$. Notice that $\Gamma_{X(K)}$ is a subquiver of a quiver of tree type $\Gamma_X$, which is a disjoint union of finitely many quivers of tree type.  For any $a\in H$, the upper order ideal $\{b\in K\mid a\preceq b\}$ is contained in $H$ since $j'$ is admissible. Thus for any $a\in K\backslash H$, we have the inclusion \begin{align}
  \{b\in K\mid b\preceq a\}\subseteq K\backslash H\subseteq H'. \label{eq: ideal one inclusion}\end{align}
Since $i'$ is admissible, for any $a\in J$, we have \begin{align}
    \{b\in K\mid b\preceq a\}&\subseteq \{b\in H\mid b\preceq a\}\cup (K\backslash H) \nonumber \\
    &\subseteq J\cup (K\backslash H) \nonumber \\
    & =H'. \label{eq:ideal inclusion}
\end{align} By the inclusions \eqref{eq: ideal one inclusion} and \eqref{eq:ideal inclusion}, the monomorphism $i=\tau_{(K\backslash H)\cup J, (K\backslash H)\cup J}$ from $A$ to $B$ is admissible. The inclusion~\eqref{eq:ideal inclusion} also implies that for any $a\in J$, we have  \begin{align}
    \{b \in H'\mid b\preceq a\}\subseteq \{b\in K\mid b\preceq a\} \subseteq H'. \label{eq: ideal another inclusion}
\end{align}
This inclusion\eqref{eq: ideal another inclusion} shows that the epimorphism $j=\beta_{(K\backslash H)\cup J, (K\backslash H)\cup J}$ is admissible. By a similar discussion as the proof of Theorem~\ref{thm:proto-abelian_D}, the diagram \[
 \begin{tikzcd}
        A \arrow[r, hook,"i"] \arrow[d, two heads, "j"] & B \arrow[d, two heads, "j'"] \\
        C \arrow[r, hook, "i'"] & D
        \end{tikzcd}
        \]
        is a pullback square. Since $H=(K\backslash H')\cup J$, this diagram is also a pushout square.  In conclusion, the conditions (PE3)-(PE5) hold. The pair $(\mathcal{M},\mathcal{P})$ gives a proto-exact structure on $\mathcal{S}_X$.
\end{proof}


By \cite{DK}, the proto-exact structure defined by $(\mathcal{M}, \mathcal{P})$ can be used to define an associative Hall algebra for $\mathcal{S}_X$. We now describe the admissible proto-exact sequences in this proto-exact structure on $\mathcal{S}_X$. \par
An admissible proto-exact sequence $A\overset{i}{\to}B\overset{j}{\to}C$ in the proto-exact category $\mathcal{S}_X$ is a sequence with $i \in \mathcal{M}, p \in \mathcal{P}$ which can be completed to  
a biCartesian diagram of the form: \[
 \begin{tikzcd}
        A \arrow[r, hook,"i"] \arrow[d, two heads, "0"] & B \arrow[d, two heads, "p"] \\
        \mathbf{0} \arrow[r, hook, "0"] & C.
        \end{tikzcd}
        \]


 If we take $B=X(H)$ for a subset $H\subseteq I$ in the biCartesian diagram above, there exist subsets $J,K\subseteq H$ such that $A=X(J)$ and $C=X(K)$ by Definition~\ref{def:admissible_in_S_X}. Precisely, for subobjects $X(J)$, $X(H)$ and $X(K)$ of $X$, there exists an admissible proto-exact sequence of the form
    \begin{align*}
        \Gamma_{X(J)} \hookrightarrow  \Gamma_{X(H)} \twoheadrightarrow  \Gamma_{X(K)} 
    \end{align*}
   in $\mathcal{S}_X$ if and only if the following conditions are satisfied:\begin{enumerate}
       \item The vertex set of $\Gamma_{X(H)}$ is the disjoint union of the vertex sets of $\Gamma_{X(J)}$ and $\Gamma_{X(K)}$. In other words, $H=J\sqcup K$.
       \item For any $a\in J$, we have $\{b\in H\mid b\preceq a\}\subseteq J$.
       \item For any $a\in K$, we have $\{b\in H\mid a\preceq b\}\subseteq K$.
   \end{enumerate}
   We call such a sequence an \textbf{basic admissible proto-exact sequence} .\par
   
   Given a sequence $A\hookrightarrow B\twoheadrightarrow C$ in $\mathcal{S}_X$, we decompose $B=\bigsqcup _{i=1}^tB_i$ into a disjoint union of connected quivers $B_i$. This sequence is \textbf{admissible proto-exact sequence} in $\mathcal{S}_X$ if and only if we have decompositions \begin{align}
   A=\bigsqcup _{i=1}^t(\bigsqcup^{l_i}_{j=1} A_{i,j})  \text{ and }  C=\bigsqcup _{i=1}^t(\bigsqcup^{l'_i}_{j=1} C_{i,j})  \nonumber 
   \end{align}
   such that \begin{align}
       \bigsqcup^{l_i}_{j=1} A_{i,j} \hookrightarrow B_i \twoheadrightarrow \bigsqcup^{l'_i}_{j=1} C_{i,j} \nonumber
   \end{align}
   is a basic admissible proto-exact sequence for all $1\leq i\leq t$ and $A\hookrightarrow B\twoheadrightarrow C$ is the direct sum of these sequences.

Just as in the proto-abelian case, one can define an equivalence relation on admissible proto-exact sequences with fixed end terms. We use the same notation   $\operatorname{Ext}_{\mathcal{E}}(A',A)_B$ and $\operatorname{Ext}_{\mathcal{E}}(A',A)$ as in Section~\ref{sec:root-independednt_subcat}. 

\begin{lemma}\label{lemma:S_X finitary}
The category $\mathcal{S}_X$ is finitary. 
\end{lemma}

\begin{proof}
The morphism sets are finite since a fixed quiver only has finitely many subquivers. For any two subquivers $P, Q$ of $\Gamma_X$, there exists finitely many basic admissible proto-exact sequences starting from $P$ and ending at $Q$ in $\mathcal{S}_X$ with middle term a subquiver of $\Gamma_X$. Since each object of $\mathcal{S}_X$ is a finite disjoint union of subquivers of $\Gamma_X$, the extension sets are also finite.  
\end{proof}

\begin{definition} \label{def:hall_algebra_S_X}
    Fix an object $X\in \mathscr{D}$ such that the underlying graph of $\Gamma_X$ is a tree. We consider the vector space spanned by the set of isomorphism classes $\{[\Gamma]\}_{\Gamma\in \mathcal{S}_X}$ with the following multiplication rule: For two  objects $\Gamma_1$ and $\Gamma_2$ of $\mathcal{S}_X$,
    \begin{align}
        [\Gamma_1]\cdot [\Gamma_2]:= \sum_{K \in \text{Iso}(\mathcal{S}_X)} | \operatorname{Ext}_{\mathcal{S}_X}(\Gamma_2,\Gamma_1)_K| [K].
    \end{align}
    Here $K$ ranges over each 
    isomorphism
    class once. 
    The associative algebra $\mathcal{H}(\mathcal{S}_X)$ obtained is called the \textbf{Hall algebra} of the subobject category $\mathcal{S}_X$.
\end{definition}

The well-definedness and associativity of Definition~\ref{def:hall_algebra_S_X} follows from Theorem~\ref{thm:proto-exact_S_X} by the general machinery of  Hall algebras of finitary proto-exact category developed in \cite{DK}. 

\begin{remark} \label{rem:graded_co_comm_Hopf_S_X}
    If we use the disjoint union to define the comultiplication, we have a graded, connected, and co-commutative Hopf algebra structure on $\mathcal{H}(\mathcal{S}_X)$.  Precisely, for any $\Gamma\in \mathcal{S}_X$, we define the comultiplication by \begin{align}
        \Delta([\Gamma])=\sum_{ \Gamma_1\sqcup \Gamma_2 = \Gamma }[\Gamma_1]\otimes [\Gamma_2] \nonumber.
    \end{align}
    The compatibility of the comultiplication and the multiplication  can be shown by an argument similar to that in  \cite[Section 6]{quiverf1}.
\end{remark}

\subsection{Isomorphic Hall algebras}\label{subsection isomorphic hall algebras}

Now we show the isomorphism between the Hall algebra of $\mathcal{S}_X$ and the Hall algebra $\mathbf{H}(\Gamma_X)$. Before constructing the isomorphism, we introduce a key lemma for quiver representations over $\mathbb{F}_1$. This property only appears for quiver representations over $\mathbb{F}_1$ 
and is completely different from the usual situation over a field.
\begin{lemma} \label{lem:decomposition_of_sequences}
    We fix a quiver $Q$ of tree type and a proto-exact sequence \begin{align}
        K\overset{\tau}{\hookrightarrow} M \overset{j}{\twoheadrightarrow} N \label{eq: proto-exact in rep}
    \end{align}
    for representations of $Q$ over $\mathbb{F}_1$. Suppose that 
\begin{align}
M \cong \bigoplus_{i=1}^t M_i \nonumber
\end{align}
is a decomposition into indecomposable modules. 
Then there exist decompositions
\begin{align}
K \cong \bigoplus_{i=1}^t(\bigoplus^{l_i}_{j=1} K_{i,j}) \nonumber
\text{ and }
N \cong \bigoplus_{i=1}^t(\bigoplus^{l'_i}_{j=1} N_{i,j}) \nonumber
\end{align}
such that the proto-exact sequence~\eqref{eq: proto-exact in rep} decomposes as a direct sum of proto-exact sequences
\begin{align}
0 \to \bigoplus^{l_i}_{j=1} K_{i,j} \to M_i \to \bigoplus^{l'_i}_{j=1} N_{i,j} \to 0 . \label{eq:sum of proto-exact seqs}
\end{align}
\end{lemma}

\begin{proof}
    For any vertex $v$ of $Q=(Q_0,Q_1)$, we denote the $\mathbb{F}_1$-vector space at this vertex of a representation $R$ of $Q$ by $R_v$. We have \[M_v\backslash\{0\}=\bigsqcup_{i=1}^t((M_i)_v\backslash \{0\}).\]
    Since $Q$ is of tree type, by Theorem 5 in \cite{quiverf1}, the dimension $\dim (M_i)_v$ is either $0$ or $1$ for any $1\leq i\leq t$ and $v\in Q_0$. We denote the full subquiver of $Q$ such that $\dim (M_i)_v=1$ on vertices of this subquiver by $Q(M_i)$. \par
    Fix $1\leq i \leq t$. We take $N_i=j(M_i)$, which is a sub-representation of $N$. Like the proof of Theorem 5 in \cite{quiverf1}, we have $N\cong N_i \oplus N/N_i$ as representation of $Q$. Moreover, $N_i$ can be identified with a full subquiver $Q(N_i)$ of $Q(M_i)$ since $j$ is an epimorphism. Take the full subquiver $Q(K_i)$ of $Q(M_i)$ with vertex set $Q(M_i)_0\backslash Q(N_i)_0$. This subquiver $Q(K_i)$ corresponds to a sub-representation $\tilde{K_i}$ of $M_i$. We have a sub-representation $K_i:=\tau^{-1}(\tilde{K_i})$ of $K$ as in the proof of Theorem 5 in \cite{quiverf1}. Moreover, $K\cong K_i\oplus (K/K_i)$. Note that the proof of  \cite[Lemma 3.2]{FRY} gives a similar structure. \par
    By carrying out this procedure for each $i$, we obtain a family of proto-exact sequences
\begin{align}
K_i \hookrightarrow M_i \twoheadrightarrow N_i . \nonumber
\end{align}    
For each $1\leq i\leq t$, take decompositions into indecomposable modules \begin{align}
N_i \cong \bigoplus_{j=1}^{l'_i} N_{i,j} \text{ and } K_i \cong \bigoplus_{j=1}^{l_i} K_{i,j}.\nonumber
\end{align}
We have the proto-exact sequence~\eqref{eq: proto-exact in rep} is given by the direct sum of the proto-exact sequences~\eqref{eq:sum of proto-exact seqs} above for all $1\leq i\leq t$.
\end{proof}

\begin{remark}\label{remark:difference F_1 and field}
    Lemma~\ref{lem:decomposition_of_sequences} does not work for quiver representations over a field. For example, we consider the representation of the quiver $Q: 1\to 2\to 3$ and denote the indecomposable module of $Q$ with dimension vector $(\dim M_1,\dim M_2, \dim M_3)$ by $M_{\dim M_1,\dim M_2, \dim M_3}$. The non-split exact sequence \begin{align}
        M_{0,1,1}\hookrightarrow M_{0,1,0}\oplus M_{1,1,1} \twoheadrightarrow M_{1,1,0} \nonumber
    \end{align}
    can not be decomposed into a direct sum of nonzero exact sequences. \par
    This is one of many differences between representation categories over the usual finite fields $\mathbb{F}_q$ and over $\mathbb{F}_1$. 
    We already mentioned that a connected quiver has finitely many indecomposable representations over $\mathbb{F}_1$, up to isomorphism, if and only if it is a tree. This is of course drastically different to the case of $\mathbb{F}_q$-representations, where Gabriel's theorem states that a connected quiver has finitely many indecomposable representations  up to isomorphism if and only if it is an orientation of a Dynkin diagram of type $ADE$.  
    Further, contrary to the case over $\mathbb{F}_q$, the Auslander-Reiten theory does not work directly for quiver representations over $\mathbb{F}_1$, and the behavior of projective modules is rather surprising. See \cite{FRY,JunSistko, Sistko26} for more discussions on these differences.
\end{remark}
\begin{theorem} \label{thm:iso_hall_algebras}
     Fix an object $X\in \mathscr{D}$ such that the underlying graph of $\Gamma_X$ is a tree. The Hall algebra $\mathcal{H}(\mathcal{S}_X)$ is isomorphic to the Hall algebra $\mathbf{H}(\Gamma_X)$ of the category of (nilponent) representations of $\Gamma_X$ over $\mathbb{F}_1$.
\end{theorem}

\begin{proof}
    Theorem~\ref{thm:bij_indec_reps_to_subobjects} implies the existence of a bijection $\Psi$ from the set of equivalence classes of indecomposable objects in $\mathcal{S}_X$ to the set of equivalence classes of indecomposable representations of $\Gamma_X$ over $\mathbb{F}_1$ such that $\dim _i(\Psi(Q))=\delta_Q(i)$ for any subquiver $Q$ of $\Gamma_X$. Define $\Psi(\bigsqcup_i \Gamma_i)$ to be $\bigoplus_i(\Psi(\Gamma_i))$. The category $Rep(\Gamma_X,\mathbb{F}_1)$ is proto-abelian by Section 6 in \cite{quiverf1}. Theorem 8 in \cite{quiverf1} implies that for any connected subquiver $\Gamma$ of $\Gamma_X$, a sequence \begin{align}
        \Psi(\Gamma_1)\hookrightarrow\Psi(\Gamma)\twoheadrightarrow \Psi(\Gamma_2) \nonumber
    \end{align} 
    is a proto-exact sequence in $Rep(\Gamma_X,\mathbb{F}_1)$ if only if the sequence \begin{align}
        \Gamma_1\hookrightarrow\Gamma\twoheadrightarrow \Gamma_2 \nonumber
    \end{align} is admissible proto-exact in $\mathcal{S}_X$. By Lemma~\ref{lem:decomposition_of_sequences}, any proto-exact sequence in $Rep(\Gamma_X,\mathbb{F}_1)$ can be decomposed into a direct sum of proto-exact sequences each of which has an indecomposable middle term. Comparing this with Definition~\ref{def:admissible_in_S_X}, we conclude that $\Psi$ sends admissible proto-exact sequences in $\mathcal{S}_X$ to proto-exact sequences induced by the proto-abelian structure in $Rep(\Gamma_X,\mathbb{F}_1)$ bijectively. In this way, Definition~\ref{def:hall_algebra_S_X} is just the reformulation of the definition of the Hall algebra $\mathbf{H}(\Gamma_X)$ using the proto-abelian structure in \cite{quiverf1} via the bijection $\Psi$.
\end{proof}


\begin{remark}\label{remark: not equivalent}
    Despite the isomorphism of Hall algebras, the categories $\mathcal{S}_X$ and $Rep(\Gamma_X, \mathbb{F}_1)$ are not equivalent. In $\mathcal{S}_X$, for any subquiver $P$ of a quiver $Q$, there always exist maps \begin{align}
     r_P= Q \hookleftarrow P \hookrightarrow P \text{ and } i_P= P \hookleftarrow P \hookrightarrow Q .\nonumber  
    \end{align}
    We have $r_P\circ i_P= \id_P \in \Mor(P,P)$. In other words, any such epimorphism $r_P$ has a retraction. This is not the case in $Rep(\Gamma_X, \mathbb{F}_1)$. For example, there exists no retraction from $\mathbb{F}_1\overset{\id}{\to} \mathbb{F}_1\to 0$ to $\mathbb{F}_1\overset{\id}{\to} \mathbb{F}_1\overset{\id}{\to} \mathbb{F}_1$ in the category of representations of the quiver $1\to 2\to 3$.
\end{remark}

\subsection{Flips and Reflections}
We show that we can give an explanation of flips using reflection of the root configuration quiver with the help of the correspondence between subword complexes and root configuration quivers defined in Section~\ref{subsection isomorphic hall algebras}. \par

 In this subsection, for any 
 $i\in I$, we use $i'$ to represent the unique index in $[n]\backslash I$ such that $\{i,i'\}$ is a flippable flat of $X$. 
The interval $[i,i']$ (if $i<i'$) or $[i',i]$ (if $i'<i$) is called the \textbf{reflection interval} of the flip $\{i,i'\}$. The complementary of the reflection interval in $[n_X]$ is called the \textbf{stable interval} of the flip $\{i,i'\}$. The stable interval is the disjoint union of two (possibly empty) intervals. According to their positions relative to the reflection interval, we call them the \textbf{left stable interval} or the \textbf{right stable interval} of the flip $\{i,i'\}$.

\begin{definition}
    A vertex $i$ of $\Gamma_X$ is called \textbf{special} if 
    all adjacent vertices of $i$ in $\Gamma_X$ are either all in the reflection interval of the flip $\{i,i'\}$ or all in the stable interval of the flip $\{i,i'\}$.
\end{definition}

\begin{theorem} \label{thm:flips_quivers}
  Assume that  $X=(W,Q,\pi,I)\in \mathscr{D}_{fin}$ with $\pi = w_0$ 
    is flipped to  $Y$ by the flat $\{i,i'\}$ for a special vertex $i$. Then $\Gamma_Y$ is obtained from $\Gamma_X$ by reverting all arrows adjacent to $i$ and relabeling the vertex $i$ of $\Gamma_X$ by $i'$.
\end{theorem}

\begin{proof}
     Suppose that $I_X= J\cup \{i\}$ and $I_Y=J\cup \{i'\}$. Denote the length of $Q_X$ and $Q_Y$ by $n$. We first suppose that $i< i'$. Using \cite[Lemma 2.6]{BC}, we have $\rx_X(i')=\rx_X(i)$. 
     The root function of $Y$ can be deduced from the root function of $X$ as follows: \begin{align}
        \rx_Y(j)=\begin{cases}
            s_{X,i}(\rx_X(j)) & \text{if } i<j\leq i' \\
            \rx_X(j) & \text{else,}
        \end{cases} \nonumber
    \end{align}
     where $s_{X,i}$ is the reflection orthogonal to the root $\rx_X(i)$. In particular, we have $\rx_Y(i')=-\rx_X(i)$. \par
     Fix two vertices $j_1,j_2\in J$. If one of $j_1$ and $j_2$ is not adjacent to $i$ in $\Gamma_X$, we have \begin{align}
         \langle \rx_Y(j_1),\rx_Y(j_2)\rangle= \langle \rx_X(j_1),\rx_X(j_2)\rangle, \nonumber
     \end{align}
     since the reflection $s_{X,i}$ preserves the bilinear form $\langle\cdot , \cdot \rangle$ and $\langle \rx_X(j),\rx_X(i)\rangle=0$ for any $j$ not adjacent to $i$ in $\Gamma_X$. Thus $\Gamma_Y$ coincides with $\Gamma_X$ outside the neighborhood (the full subquiver of $\Gamma_X$ containing $i$ and its adjacent vertices) of $i$.\par
     Now if $j_1$ and $j_2$ are both adjacent to $i$, we have \begin{align}
         \langle \rx_Y(j_1),\rx_Y(j_2)\rangle=\begin{cases}
       \langle \rx_X(j_1),\rx_X(j_2)\rangle  & \text{if } j_1,j_2\in [i+1,i']\\
       \langle s_{X,i}(\rx_X(j_1)),s_{X,i}(\rx_X(j_2))\rangle  & \text{if } j_1,j_2\in [n]\backslash [i+1,i'],
         \end{cases} \nonumber
     \end{align}
     since $i$ is a special vertex. Since the reflection $s_{X,i}$ preserves the bilinear form $\langle\cdot , \cdot \rangle$, we have $\Gamma_Y$ coincides with $\Gamma_X$ for the adjacent vertices of $i$.\par
     Now we consider the arrow between $j_1$ and $i$. We have \begin{align}
        \langle \rx_Y(j_1),\rx_Y(i')\rangle=\begin{cases}
             \langle\rx_X(j_1),-\rx_X(i)\rangle & \text{if } j_1\in [n]\backslash [i,i'] \\
              \langle s_{X,i}(\rx_X(j_1)),-\rx_X(i)\rangle= \langle\rx_X(j_1),\rx_X(i)\rangle& \text{if } j_1\in  [i,i'].
         \end{cases} \nonumber
     \end{align}
     In either case, we have \begin{align}
     [(j_1-i')\langle\rx_Y(j_1),-\rx_Y(i')\rangle] \cdot  [(j_1-i)\langle\rx_X(j_1),-\rx_X(i)\rangle]<0.  \nonumber
     \end{align}
     Thus $\Gamma_Y$ is obtained from $\Gamma_X$ by reverting all arrows adjacent to $i$ and relabeling the vertex $i$ of $\Gamma_X$ by $i'$. 
     
     The case $i>i'$ 
     is proved by a similar argument.
\end{proof}

The change of orientation at a special vertex does not change the subquiver category up to isomorphism.
\begin{proposition}\label{prop: flip and reflection}
    Under the assumption of Theorem~\ref{thm:flips_quivers}, the subquiver category $\mathcal{S}_X$ and $\mathcal{S}_Y$ are isomorphic.
\end{proposition}
\begin{proof}
    Objects of $\mathcal{S}_X$ and $\mathcal{S}_Y$ are parameterized by $P(I_X)^{\mathbb{N}}$ and $P(I_Y)^{\mathbb{N}}$. Here $P(I_X)$ (respectively, $P(I_Y)$) is the power set of $I_X$ (respectively, $I_Y$). The bijection \begin{align}
        j\to \begin{cases}
            j & \text{ if } j\neq i \nonumber \\
            i'& \text{ if } j=i       \nonumber    
        \end{cases}
    \end{align}
    induces a bijection from $\operatorname{Ob}(\mathcal{S}_X)$ to $\operatorname{Ob}(\mathcal{S}_Y)$. Morphisms in the subquiver category are given by label preserving partial isomorphisms of quivers, which corresponds to intersection of the vertex sets. Thus this bijection of quivers extends to an isomorphism of category.
\end{proof}
Given an
object $X=(W,Q,\pi,I)\in \mathscr{D}_{fin}$ such that $\pi = w_0$ and $\Gamma_X$ is a tree, we have an isomorphism of Hall algebras of $\mathcal{S}_X$ and $Rep(Q, \mathbb{F}_1)$ by Theorem~\ref{thm:iso_hall_algebras}. Although changing the orientation does not change the category up to isomorphism, different orientations induce different proto-exact structures. We study how changing the orientation of a quiver affects the Hall algebra using flips in the subword complex. Any special vertex $i$ gives a flip with respect to the flat $\{i,i'\}$. When we consider the Hall algebra, we have isomorphic algebras.

\begin{corollary}\label{coro 3.30}
    If $\Gamma_X$ is of tree type, for any special vertex $i$ of $\Gamma_X$, the flip $X\overset{\{i,i'\}}{\to} Y$ induces an isomorphism of algebras from $\mathcal{H}(\mathcal{S}_X)$ to $\mathcal{H}(\mathcal{S}_Y)$.
\end{corollary}
\begin{proof}
By Remark~\ref{rem:graded_co_comm_Hopf_S_X} and the isomorphism given in Theorem~\ref{thm:flips_quivers}, $\mathcal{H}(\mathcal{S}_X)$ is thus a graded, connected, and co-commutative Hopf algebra. 
By the Milnor--Moore theorem, it is the universal enveloping algebra 
of the graded Lie algebra $\mathfrak{n}_{\Gamma_X}$ of its primitive elements.\par
It follows from the definition of the coproduct that 
$[M] \in \mathcal{H}(\mathcal{S}_X)$ is primitive iff $M$ is given by an indecomposable representation of $\Gamma_X$ via the isomorphism given in Theorem~\ref{thm:iso_hall_algebras}.\par
Using the proof of \cite[Proposition 6.3]{JunSistko}, the flip $X\overset{\{i,i'\}}{\to} Y$ induces an isomorphism of Lie algebras between $\mathfrak{n}_{\Gamma_X}$ and $\mathfrak{n}_{\Gamma_Y}$. This induces an isomorphism between Hall algebras $\mathcal{H}(\mathcal{S}_X)$ and $\mathcal{H}(\mathcal{S}_Y)$.

\end{proof}

\subsection{Type A}

In this section, we apply the results above to the case $A_n$, in which we can give a realization of the universal enveloping algebra $U(\mathfrak{n}_+)$. 


\begin{corollary} \label{cor:type_A_iso}
    For an object $X\in \mathscr{D}$ such that $W_X$ is the Weyl group of type $A_n$ and $\Gamma_X$ is the Dynkin quiver of type $A_n$, we have an isomorphism of algebra $\psi_X: \mathcal{H}(\mathcal{S}_X)\to U(\mathfrak{n}_+)$.
\end{corollary}
\begin{proof}
    By Theorem~\ref{thm:iso_hall_algebras}, the Hall algebra $\mathcal{H}(\mathcal{S}_X)$ is isomorphic to the Hall algebra $\mathbf{H}(\Gamma_X)$. By Theorem~\ref{thm:univ_enveloping}, we have the desired isomorphism.
\end{proof}

Corollary~\ref{cor:type_A_iso} implies that we can get an integral basis for $U(\mathfrak{n}_+)$ using direct sum of subobjects of such an $X$. By construction, the product of two elements of this basis is a linear combination of other basis elements with positive integer coefficients. In this way, we construct an explicit realization of $U(\mathfrak{n}_+)$ in type $A_n$.

\begin{definition}\label{def:standard object}
    Fix $n\in \mathbb{N}$, we choose the following object $X$:
    \begin{enumerate}
        \item $W_X\simeq S_{n+1}$;
        \item The word $Q_X=s_1\cdots s_ns_1\cdots s_ns_1\cdots s_{n-1}\cdots s_1$;
        \item $I_X=[n]$;
        \item $\pi_X=w_0$ is the longest element in $W_X$.
    \end{enumerate}
\end{definition}

 We call it the \textbf{standard object} of type $A_n$. The root configuration quiver associated to $X$ is $\Gamma_X=1\rightarrow \cdots \rightarrow n$. Now we fix a standard object $X$ of type $A_n$. 

\begin{example} \label{ex:standard_A_3}
    Take $n=3$. The standard $X$ of type $A_3$ is of the form \begin{align}
        X=(S_4, s_1s_2s_3s_1s_2s_3s_1s_2s_1, s_1s_2s_3s_1s_2s_1, \{1,2,3\}). \nonumber
    \end{align}
    The indecomposable objects of $\mathcal{S}_X$ and certain multiplication rules are listed as follows:
\begin{align*}
\{2\} \cdot \{1\} &= \{1\} + \{2\} + \{1,2\}, & \{1\} \cdot\{2\} &= \{1\} + \{2\}, \\[4pt]
\{3\} \cdot \{2\} &= \{2\} + \{3\} + \{2,3\}, & \{2\} \cdot \{3\} &= \{2\} + \{3\}, \\[4pt]
\{1\} \cdot \{3\} &= \{3\} \cdot \{1\}=\{1\} + \{3\}\\[4pt]
\{3\} \cdot \{1,2\} &= \{1,2\} + \{3\} + \{1,2,3\}, &
\{1,2\} \cdot \{3\} &= \{1,2\} + \{3\}, \\[4pt]
\{1\} \cdot \{2,3\} &= \{1\} + \{2,3\} + \{1,2,3\}, &
\{2,3\} \cdot \{1\} &= \{1\} + \{2,3\}.
\end{align*}
Here, we use a set $J$ in $\{\{1\},\{2\},\{3\},\{1,2\},\{2,3\},\{1,2,3\}\}$ to represent the corresponding subquiver $\Gamma_{X(J)}$ of $\Gamma_X$.
\par
Using notations in \cite{Zijun},
the canonical gallery $g_X$ of $X$ is given by \begin{align}
    g_X=([-\alpha_1,-\alpha_2, -\alpha_3, \alpha_1 ,\alpha_1+\alpha_2, \alpha_1+\alpha_2+\alpha_3, \alpha_2, \alpha_2+\alpha_3,\alpha_3],\{1,2,3\}). \nonumber
\end{align}
Any set of simple roots determines a subobject $Z$ of $X$, where $W_{Z}$ is the parabolic subgroup of $W_X$ generated by simple reflections corresponding to simple roots in this set. \par 

For example, if we take the set $\{\alpha_2,\alpha_3\}$, then the projection of $g_X$ to the vector space spanned by $\{\alpha_2,\alpha_3\}$ yields a folded gallery \begin{align}
    g=([-\alpha_2, -\alpha_3, \alpha_2, \alpha_2+\alpha_3, \alpha_3],\{1,2\}). \nonumber
\end{align}
This gallery is the canonical gallery of $Z$.\par

The vertex $3$ is a special vertex of $\Gamma_X$. Notice that $s_3s_1s_2s_3s_1s_2$ is a reduced expression of $\pi_X$. We have $\{3,9\}$ is a flippable flat of $X$ and we have \begin{align}
    X \overset{\{3,9\}}{\to} Y=(S_4, s_1s_2s_3s_1s_2s_3s_1s_2s_1, s_3s_1s_2s_3s_1s_2, \{1,2,9\}). \nonumber
\end{align} 
The root configuration quiver $\Gamma_Y$ is $1\rightarrow 2\leftarrow 9$. \par

The indecomposable objects of $\mathcal{S}_Y$ and certain multiplication rules are listed as follows:
\begin{align*}
\{2\} \cdot \{1\} &= \{1\} + \{2\} + \{1,2\}, & \{1\} \cdot \{2\} &= \{1\} + \{2\}, \\[4pt]
\{9\} \cdot \{2\} &= \{2\} + \{9\}, & \{2\} \cdot \{9\} &= \{2\} + \{9\}+\{2,9\}, \\[4pt]
\{1\} \cdot \{9\} &= \{9\} \cdot \{1\}=\{1\} + \{9\}\\[4pt]
\{9\} \cdot \{1,2\} &= \{1,2\} + \{9\}, &
\{1,2\} \cdot \{9\} &= \{1,2\} + \{9\}+\{1,2,9\}, \\[4pt]
\{1\} \cdot \{2,9\} &= \{1\} + \{2,9\}, &
\{2,9\} \cdot \{1\} &= \{1\} + \{2,9\}+\{1,2,9\}.
\end{align*}
Here, we use a subset $J$ of $\{1,2,9\}$ to represent the corresponding subquiver $\Gamma_{Y(J)}$ of $\Gamma_Y$. Subobjects of $Y$ gives another realization of $U(\mathfrak{n_+)}$. \par

\end{example}

\subsection{Examples of non-linear orientations}

In this section we show more examples. We use notation analogous to that in Example~\ref{ex:standard_A_3}.

\begin{example} \label{example D_4 standard}
    Consider the Coxeter group $W(D_4) \cong (\mathbb{Z}_2)^3 \rtimes S_4$ of type $D_4$ with the Coxeter graph 
   
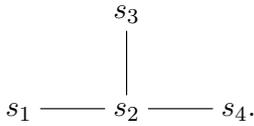
\xymatrix{
 & s_3 \\
s_1 \ar@{-}[r] & s_2 \ar@{-}[u] \ar@{-}[r] & s_4.
}\par
      Take an object $X$ of the form \begin{align}
        X=(W(D_4), s_1s_2s_3s_3s_4s_1, s_3s_1, \{1,2,4,5\}). \nonumber
    \end{align}
    The root function of $X$ is given by \begin{align}
       & \rx_X(1)=\alpha_1, \rx_X(2)=\alpha_2, \rx_X(3)=\alpha_3, \nonumber \\
        & \rx_X(4)=-\alpha_3, \rx_X(5)=\alpha_4, \nonumber \\
        & \rx_X(6)=\alpha_1\nonumber
    \end{align}
The root configuration quiver $\Gamma_X$ is 
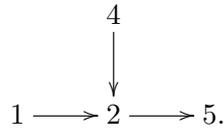
\xymatrix{
 & 4 \ar [d] \\
1 \ar[r] & 2  \ar[r] & 5.
}\par

Now $1$ is a special vertex of $X$. If we flip $X$ to $Y=(W(D_4), s_1s_2s_3s_3s_4s_1, s_3s_1, \{2,4,5,6\})$, the root function of $Y$ is
\begin{align}
       & \rx_Y(1)=\alpha_1, \rx_Y(2)=\alpha_1+\alpha_2, \rx_Y(3)=\alpha_3, \nonumber \\
        & \rx_Y(4)=-\alpha_3, \rx_Y(5)=\alpha_4, \nonumber \\
        & \rx_Y(6)=-\alpha_1\nonumber
    \end{align}
In this case, the root configuration quiver $\Gamma_Y$ is  
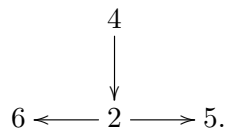
\xymatrix{
 & 4 \ar [d] \\
6 & 2 \ar[l] \ar[r] & 5.
}\par
\end{example}
In general, the underlying graph of the root configuration quiver can be different from the Coxeter graph.


\begin{example}\label{example D_4}
    Take an object $X$ of the form \begin{align}
        X=(W(D_4), s_1s_2s_2s_3s_4s_1, s_2s_1, \{1,2,4,5\}). \nonumber
    \end{align}
    The root function of $X$ is given by \begin{align}
       & \rx_X(1)=\alpha_1, \rx_X(2)=\alpha_2, \rx_X(3)=\alpha_2, \nonumber \\
        & \rx_X(4)=\alpha_2+\alpha_3, \rx_X(5)=\alpha_2+\alpha_4, \nonumber \\
        & \rx_X(6)=\alpha_1+\alpha_2\nonumber
    \end{align}
The root configuration quiver $\Gamma_X$ is 
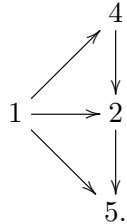
\xymatrix{
 & 4 \ar [d] \\
1 \ar[r] \ar[dr] \ar[ur] & 2  \ar[d] \\
& 5.
}\par
\end{example}
Here is an example for an infinite Coxeter group. 
\begin{example}\label{example A_2 inf}
    Consider the affine Coxeter group $W(\widetilde{A}_2)$ of type $\widetilde{A}_2$ with Coxeter graph
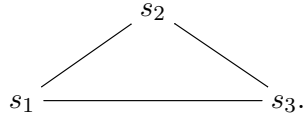
\[
\xymatrix@C=3em@R=2em{
& s_2 \ar@{-}[dl] \ar@{-}[dr] & \\
s_1 \ar@{-}[rr] && s_3.
}
\]
Take an object $X$ of the form \begin{align}
        X=(W(\widetilde{A}_2), s_1s_2s_2s_3s_1s_2, s_2s_1s_2, \{1,3,4\}). \nonumber
    \end{align}
    We use the Tits root system to construct a three-dimensional vector space $V_X=\mathbb{R}\alpha_1\oplus \mathbb{R}\alpha_2 \oplus \mathbb{R}\alpha_3$.
    The root function of $X$ is given by \begin{align}
       & \rx_X(1)=\alpha_1, \rx_X(2)=\alpha_2, \rx_X(3)=-\alpha_2, \nonumber \\
        & \rx_X(4)=\alpha_2+\alpha_3, \rx_X(5)=\alpha_1+\alpha_2, \nonumber \\
        & \rx_X(6)=\alpha_1\nonumber
    \end{align}
The root configuration quiver $\Gamma_X$ is 
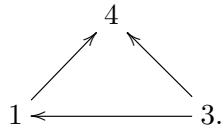
\xymatrix{
 & 4  \\
1 \ar[ur] & & 3.  \ar[ul] \ar[ll] 
}\par

\end{example}

\section*{Acknowledgements}
We are very grateful to Tobias Dyckerhoff, Stéphane Gaussent, Paul Philippe, and Petra Schwer for many useful discussions. Z.Li thanks Stéphane Gaussent for supervision and Petra Schwer for advising the study of subword complexes. M.Gorsky acknowledges support by the Deutsche Forschungsgemeinschaft SFB 1624 ``Higher structures, moduli spaces and integrability'' (506632645). Z.Li acknowledges support from the Procope project ``Buildings, galleries and beyond'' and the PHD funding from Institut Camille Jordan, Université Jean Monnet.

\bibliographystyle{alpha}
\bibliography{references, biblio} 

@article{casals_gao,
    AUTHOR = {Casals, Roger and Gao, Honghao},
     TITLE = {A {L}agrangian filling for every cluster seed},
   JOURNAL = {Invent. Math.},
  FJOURNAL = {Inventiones Mathematicae},
    VOLUME = {237},
      YEAR = {2024},
    NUMBER = {2},
     PAGES = {809--868},
      ISSN = {0020-9910,1432-1297},
   MRCLASS = {53D12 (13F60 57K33)},
  MRNUMBER = {4768635},
       DOI = {10.1007/s00222-024-01268-y},
       URL = {https://doi.org/10.1007/s00222-024-01268-y},
}

@article{G_subword_3,
  title={Subword Complexes and Nil-{H}ecke Moves},
  author={Gorsky, Mikhail},
  journal={Modeling and Analysis of Information Systems},
  volume={20},
  number={6},
  pages={121--128},
  year={2013}
}

@article{gorsky_edge,
    AUTHOR = {Gorsky, Mikhail},
     TITLE = {Subword complexes and edge subdivisions},
      NOTE = {Published in Russian in Tr. Mat. Inst. Steklova {\bf 286}
              (2014), 129--143},
   JOURNAL = {Proc. Steklov Inst. Math.},
  FJOURNAL = {Proceedings of the Steklov Institute of Mathematics},
    VOLUME = {286},
      YEAR = {2014},
    NUMBER = {1},
     PAGES = {114--127},
      ISSN = {0081-5438,1531-8605},
   MRCLASS = {20F55 (05E15 05E45)},
  MRNUMBER = {3482594},
MRREVIEWER = {T.\ Kyle\ Petersen},
       DOI = {10.1134/S0081543814060078},
       URL = {https://doi.org/10.1134/S0081543814060078},
}

@article{CGGLSS,
    AUTHOR = {Casals, Roger and Gorsky, Eugene and Gorsky, Mikhail and Le,
              Ian and Shen, Linhui and Simental, Jos\'e},
     TITLE = {Cluster structures on braid varieties},
   JOURNAL = {J. Amer. Math. Soc.},
  FJOURNAL = {Journal of the American Mathematical Society},
    VOLUME = {38},
      YEAR = {2025},
    NUMBER = {2},
     PAGES = {369--479},
      ISSN = {0894-0347,1088-6834},
   MRCLASS = {13F60 (14M15 20F36)},
  MRNUMBER = {4868947},
       DOI = {10.1090/jams/1048},
       URL = {https://doi.org/10.1090/jams/1048},
}

@article{CGGS2,
      title={Positroid Links and Braid varieties}, 
      author={Roger Casals and Eugene Gorsky and Mikhail Gorsky and José Simental},
      year={to appear, 2026.},
      url={https://arxiv.org/abs/2105.13948v4},
      journal = {J. reine angew. Math. (Crelle's Journal)},
}

@article{pilaud_santos_12,
    AUTHOR = {Pilaud, Vincent and Santos, Francisco},
     TITLE = {The brick polytope of a sorting network},
   JOURNAL = {European J. Combin.},
  FJOURNAL = {European Journal of Combinatorics},
    VOLUME = {33},
      YEAR = {2012},
    NUMBER = {4},
     PAGES = {632--662},
      ISSN = {0195-6698,1095-9971},
   MRCLASS = {52B12 (05C75 52C30)},
  MRNUMBER = {2864447},
       DOI = {10.1016/j.ejc.2011.12.003},
       URL = {https://doi.org/10.1016/j.ejc.2011.12.003},
}

@article{pilaud_stump_15,
    AUTHOR = {Pilaud, Vincent and Stump, Christian},
     TITLE = {Brick polytopes of spherical subword complexes and generalized
              associahedra},
   JOURNAL = {Adv. Math.},
  FJOURNAL = {Advances in Mathematics},
    VOLUME = {276},
      YEAR = {2015},
     PAGES = {1--61},
      ISSN = {0001-8708,1090-2082},
   MRCLASS = {05E45 (05E15 05E30 20F55 52B12)},
  MRNUMBER = {3327085},
MRREVIEWER = {Jian-yi\ Shi},
       DOI = {10.1016/j.aim.2015.02.012},
       URL = {https://doi.org/10.1016/j.aim.2015.02.012},
}

@article{cls,
    AUTHOR = {Ceballos, Cesar and Labb\'e, Jean-Philippe and Stump,
              Christian},
     TITLE = {Subword complexes, cluster complexes, and generalized
              multi-associahedra},
   JOURNAL = {J. Algebraic Combin.},
  FJOURNAL = {Journal of Algebraic Combinatorics. An International Journal},
    VOLUME = {39},
      YEAR = {2014},
    NUMBER = {1},
     PAGES = {17--51},
      ISSN = {0925-9899,1572-9192},
   MRCLASS = {05E45 (52B05)},
  MRNUMBER = {3144391},
MRREVIEWER = {Micha\l\ Adamaszek},
       DOI = {10.1007/s10801-013-0437-x},
       URL = {https://doi.org/10.1007/s10801-013-0437-x},
}

@article{escobar,
    AUTHOR = {Escobar, Laura},
     TITLE = {Brick manifolds and toric varieties of brick polytopes},
   JOURNAL = {Electron. J. Combin.},
  FJOURNAL = {Electronic Journal of Combinatorics},
    VOLUME = {23},
      YEAR = {2016},
    NUMBER = {2},
     PAGES = {Paper 2.25, 18},
      ISSN = {1077-8926},
   MRCLASS = {14M15 (05E99 14M25)},
  MRNUMBER = {3512647},
MRREVIEWER = {Justin\ Brown},
       DOI = {10.37236/5038},
       URL = {https://doi.org/10.37236/5038},
}

@book {bb_coxeter,
    AUTHOR = {Bj\"{o}rner, Anders and Brenti, Francesco},
     TITLE = {Combinatorics of {C}oxeter groups},
    SERIES = {Graduate Texts in Mathematics},
    VOLUME = {231},
 PUBLISHER = {Springer, New York},
      YEAR = {2005},
     PAGES = {xiv+363},
      ISBN = {978-3540-442387},
   MRCLASS = {05-01 (05E15 20F55)},
  MRNUMBER = {2133266},
MRREVIEWER = {Jian-yi Shi},
}

@article{CKP,
 author = {{\c{C}}anak{\c{c}}{\i}, {\.I}lke and Kalck, Martin and Pressland, Matthew},
 title = {Cluster categories for completed infinity-gons. {I}: {Categorifying} triangulations},
 fjournal = {Journal of the London Mathematical Society. Second Series},
 journal = {J. Lond. Math. Soc., II. Ser.},
 issn = {0024-6107},
 volume = {111},
 number = {2},
 pages = {31},
 note = {Id/No e70092},
 year = {2025},
 language = {English},
 doi = {10.1112/jlms.70092},
 zbMATH = {7989661}, 
 Zbl = {1565.18021}
}

@article {NakaokaPalu1,
    AUTHOR = {Nakaoka, Hiroyuki and Palu, Yann},
     TITLE = {Extriangulated categories, {H}ovey twin cotorsion pairs and
              model structures},
   JOURNAL = {Cah. Topol. G\'{e}om. Diff\'{e}r. Cat\'{e}g.},
  FJOURNAL = {Cahiers de Topologie et G\'{e}om\'{e}trie Diff\'{e}rentielle
              Cat\'{e}goriques},
    VOLUME = {60},
      YEAR = {2019},
    NUMBER = {2},
     PAGES = {117--193},
      ISSN = {1245-530X,2681-2363},
   MRCLASS = {18E10 (18E30 18E35 18G15 18G55)},
  MRNUMBER = {3931945},
MRREVIEWER = {Panyue\ Zhou},
}

@article {NakaokaPalu2,
	AUTHOR = {Hiroyuki Nakaoka and Yann Palu},
	TITLE = {External triangulation of the homotopy category of an exact quasi-category},
	journal = {arXiv preprint arXiv:2004.02479}, 
	YEAR = {2020},
}

@incollection{schiffmann3,
 author = {Schiffmann, Olivier G.},
 title = {Kac polynomials and {Lie} algebras associated to quivers and curves},
 booktitle = {Proceedings of the international congress of mathematicians 2018, ICM 2018, Rio de Janeiro, Brazil, August 1--9, 2018. Volume II. Invited lectures},
 isbn = {978-981-3272-91-0; 978-981-327-287-3; 978-981-3272-89-7},
 pages = {1393--1424},
 year = {2018},
 publisher = {Hackensack, NJ: World Scientific; Rio de Janeiro: Sociedade Brasileira de Matem{\'a}tica (SBM)},
 language = {English},
 doi = {10.1142/9789813272880_0102},
 zbMATH = {7250525},
 Zbl = {1495.17035}
}

@incollection{schiffmann2,
 author = {Schiffmann, Olivier},
 title = {Lectures on canonical and crystal bases of {Hall} algebras},
 booktitle = {Geometric methods in representation theory. II. Selected papers based on the presentations at the summer school, Grenoble, France, June 16 -- July 4, 2008},
 isbn = {978-2-85629-361-4},
 pages = {143--259},
 year = {2012},
 publisher = {Paris: Soci{\'e}t{\'e} Math{\'e}matique de France},
 language = {English},
 zbMATH = {6308124},
 Zbl = {1356.17015}
}

@incollection{schiffmann1,
 author = {Schiffmann, Olivier},
 title = {Lectures on {Hall} algebras},
 booktitle = {Geometric methods in representation theory. II. Selected papers based on the presentations at the summer school, Grenoble, France, June 16 -- July 4, 2008},
 isbn = {978-2-85629-361-4},
 pages = {1--141},
 year = {2012},
 publisher = {Paris: Soci{\'e}t{\'e} Math{\'e}matique de France},
 language = {English},
 zbMATH = {6308123},
 Zbl = {1309.18012}
}

@article {R1,
    AUTHOR = {Ringel, Claus Michael},
     TITLE = {Hall algebras and quantum groups},
   JOURNAL = {Invent. Math.},
  FJOURNAL = {Inventiones Mathematicae},
    VOLUME = {101},
      YEAR = {1990},
    NUMBER = {3},
     PAGES = {583--591},
      ISSN = {0020-9910,1432-1297},
   MRCLASS = {16E60 (16G60 16P10 16S30 17B37)},
  MRNUMBER = {1062796},
MRREVIEWER = {Jie\ Du},
       DOI = {10.1007/BF01231516},
       URL = {https://doi.org/10.1007/BF01231516},
}

@article{trinh2021hecke,
  title={From the {H}ecke category to the unipotent locus},
  author={Trinh, Minh-T{\^a}m Quang},
  journal={arXiv preprint arXiv:2106.07444},
  year={2021}
}

@misc{CGGSSBS,
 author = {Casals, Roger and Galashin, Pavel and Gorsky, Mikhail and Shen, Linhui and Sherman-Bennett, Melissa and Simental, Jos{\'e}},
 title = {Comparing cluster algebras on braid varieties},
 year = {2025},
 howpublished = {Preprint, {arXiv}:2508.03816 [math.{AG}] (2025)},
 url = {https://arxiv.org/abs/2508.03816},
 arXiv = {arXiv:2508.03816}
}

@article{GLSB,
 author = {Galashin, Pavel and Lam, Thomas and Sherman-Bennett, Melissa},
 title = {Braid variety cluster structures. {II}: {General} type},
 fjournal = {Inventiones Mathematicae},
 journal = {Invent. Math.},
 issn = {0020-9910},
 volume = {243},
 number = {3},
 pages = {1079--1127},
 year = {2026},
 language = {English},
 doi = {10.1007/s00222-025-01390-5},
 zbMATH = {8162061}
}

@incollection {Hu,
    AUTHOR = {Hubery, Andrew},
     TITLE = {From triangulated categories to {L}ie algebras: a theorem of
              {P}eng and {X}iao},
 BOOKTITLE = {Trends in representation theory of algebras and related
              topics},
    SERIES = {Contemp. Math.},
    VOLUME = {406},
     PAGES = {51--66},
 PUBLISHER = {Amer. Math. Soc., Providence, RI},
      YEAR = {2006},
      ISBN = {0-8218-3818-0},
   MRCLASS = {16G20 (17B65 18E30)},
  MRNUMBER = {2258041},
       DOI = {10.1090/conm/406/07653},
       URL = {https://doi.org/10.1090/conm/406/07653},
}

@article{BBGH,
 author = {Baillargeon, Rose-Line and Br{\"u}stle, Thomas and Gorsky, Mikhail and Hassoun, Souheila},
 title = {On the lattices of exact and weakly exact structures},
 fjournal = {Journal of Algebra},
 journal = {J. Algebra},
 issn = {0021-8693},
 volume = {612},
 pages = {77--109},
 year = {2022},
 language = {English},
 doi = {10.1016/j.jalgebra.2022.07.043},
 zbMATH = {7596227},
 Zbl = {1505.18017}
}

@misc{Zijun,
  author = {Li, Zijun},
  title = {A path model for {MV} polytopes in type ${A}_n$},
  year = {2026},
  howpublished = {Preprint, {arXiv}:2603.28634 [math.{RT}] (2026)}, 
  url = {https://arxiv.org/abs/2603.28634},
  arXiv = {arXiv:2603.28634}
}

@misc{CY,
 author = {Cooper, Benjamin and Young, Matthew B.},
 title = {Hall algebras via 2-{Segal} spaces},
 year = {2024},
 howpublished = {Preprint, {arXiv}:2409.19384 [math.{CT}] (2024)},
 url = {https://arxiv.org/abs/2409.19384},
 arXiv = {arXiv:2409.19384}
}

@misc{GCKT,
 author = {G{\'a}lvez-Carrillo, Imma and Kock, Joachim and Tonks, Andrew},
 title = {Decomposition spaces in {Combinatorics}},
 year = {2016},
 howpublished = {Preprint, {arXiv}:1612.09225 [math.{CO}] (2016)},
 url = {https://arxiv.org/abs/1612.09225},
 arXiv = {arXiv:1612.09225}
}

@article{FRY,
 author = {Fu, Changjian and Ran, Longjun and Yang, Liang},
 title = {On homological properties of the category of {{\(\mathbb{F}_1\)}}-representations over a linear quiver of type {{\(\mathbb{A}_n\)}}},
 fjournal = {Journal of Algebra},
 journal = {J. Algebra},
 issn = {0021-8693},
 volume = {646},
 pages = {1--16},
 year = {2024},
 language = {English},
 doi = {10.1016/j.jalgebra.2024.01.034},
 zbMATH = {7818589},
 Zbl = {1544.16014}
}

@article{FangGorsky,
  title={Exact structures and degeneration of {H}all algebras},
  author={Fang, Xin and Gorsky, Mikhail},
  journal={Advances in Mathematics},
  volume={398},
  pages={108210},
  year={2022},
  publisher={Elsevier}
}

@misc{DHM,
 author = {Davis, James F. and Hersh, Patricia and Miller, Ezra},
 title = {Fibers of maps to totally nonnegative spaces},
 year = {2019},
 howpublished = {Preprint, {arXiv}:1903.01420 [math.{CO}] (2019)},
 url = {https://arxiv.org/abs/1903.01420},
 arXiv = {arXiv:1903.01420}
}

@article{JahnStump,
 author = {Jahn, Dennis and Stump, Christian},
 title = {Bruhat intervals, subword complexes and brick polyhedra for finite {Coxeter} groups},
 fjournal = {Mathematische Zeitschrift},
 journal = {Math. Z.},
 issn = {0025-5874},
 volume = {304},
 number = {2},
 pages = {32},
 note = {Id/No 24},
 year = {2023},
 language = {English},
 doi = {10.1007/s00209-023-03267-w},
 zbMATH = {7684548},
 Zbl = {1518.05209}
}

@article{JunSistko23,
 author = {Jun, Jaiung and Sistko, Alexander},
 title = {On quiver representations over {{\(\mathbb{F}_1\)}}},
 fjournal = {Algebras and Representation Theory},
 journal = {Algebr. Represent. Theory},
 issn = {1386-923X},
 volume = {26},
 number = {1},
 pages = {207--240},
 year = {2023},
 language = {English},
 doi = {10.1007/s10468-021-10092-4},
 zbMATH = {7659756},
 Zbl = {1525.16011}
}

@article{JunSistko,
 author = {Jun, Jaiung and Sistko, Alexander},
 title = {Coefficient quivers, {{\(\mathbb{F}_1\)}}-representations, and {Euler} characteristics of quiver {Grassmannians}},
 fjournal = {Nagoya Mathematical Journal},
 journal = {Nagoya Math. J.},
 issn = {0027-7630},
 volume = {255},
 pages = {561--617},
 year = {2024},
 language = {English},
 doi = {10.1017/nmj.2023.37},
 zbMATH = {7927417},
 Zbl = {1561.16014}
}

@article{KM2005,
 author = {Knutson, Allen and Miller, Ezra},
 title = {Gr{\"o}bner geometry of {Schubert} polynomials},
 fjournal = {Annals of Mathematics. Second Series},
 journal = {Ann. Math. (2)},
 issn = {0003-486X},
 volume = {161},
 number = {3},
 pages = {1245--1318},
 year = {2005},
 language = {English},
 doi = {10.4007/annals.2005.161.1245},
 zbMATH = {5004653},
 Zbl = {1089.14007}
}

@book{AbramenkoBrown2008,
  author    = {Abramenko, Peter and Brown, Kenneth S.},
  title     = {Buildings: Theory and Applications},
  series    = {Graduate Texts in Mathematics},
  volume    = {248},
  publisher = {Springer},
  address   = {New York},
  year      = {2008},
  isbn      = {978-0-387-78834-0},
  doi       = {10.1007/978-0-387-78835-7}
}

@article{Schwer,
 author = {Schwer, Petra},
 title = {Shadows in the wild -- folded galleries and their applications},
 fjournal = {Jahresbericht der Deutschen Mathematiker-Vereinigung (DMV)},
 journal = {Jahresber. Dtsch. Math.-Ver.},
 issn = {0012-0456},
 volume = {124},
 number = {1},
 pages = {3--41},
 year = {2022},
 language = {English},
 doi = {10.1365/s13291-021-00244-2},
 zbMATH = {7501304},
 Zbl = {1551.20064}
}

@article{ELY,
 author = {Eberhardt, Jens Niklas and Lorscheid, Oliver and Young, Matthew B.},
 title = {Group completion in the {{\(K\)}}-theory and {Grothendieck}-{Witt} theory of proto-exact categories},
 fjournal = {Journal of Pure and Applied Algebra},
 journal = {J. Pure Appl. Algebra},
 issn = {0022-4049},
 volume = {226},
 number = {8},
 pages = {32},
 note = {Id/No 107018},
 year = {2022},
 language = {English},
 doi = {10.1016/j.jpaa.2022.107018},
 zbMATH = {7517045},
 Zbl = {1524.19003}
}

@article{BerGr,
 author = {Berenstein, Arkady and Greenstein, Jacob},
 title = {Primitively generated {Hall} algebras.},
 fjournal = {Pacific Journal of Mathematics},
 journal = {Pac. J. Math.},
 issn = {1945-5844},
 volume = {281},
 number = {2},
 pages = {287--331},
 year = {2016},
 language = {English},
 doi = {10.2140/pjm.2016.281.287},
 zbMATH = {6545112},
 Zbl = {1338.16016}
}

@book{DK,
 author = {Dyckerhoff, Tobias and Kapranov, Mikhail},
 title = {Higher {S}egal spaces},
 fseries = {Lecture Notes in Mathematics},
 series = {Lect. Notes Math.},
 issn = {0075-8434},
 volume = {2244},
 isbn = {978-3-030-27122-0; 978-3-030-27124-4},
 year = {2019},
 publisher = {Cham: Springer},
 language = {English},
 doi = {10.1007/978-3-030-27124-4},
 zbMATH = {7103772},
 Zbl = {1459.18001}
}

@article{Sistko26,
 author = {Sistko, Alexander},
 title = {On semisimple proto-abelian categories associated to inverse monoids},
 fjournal = {Journal of Algebra},
 journal = {J. Algebra},
 issn = {0021-8693},
 volume = {686},
 pages = {354--397},
 year = {2026},
 language = {English},
 doi = {10.1016/j.jalgebra.2025.07.052},
 zbMATH = {8117663}
}

@misc{Mozgovoy,
 author = {Mozgovoy, Sergey},
 title = {Proto-exact and parabelian categories},
 year = {2025},
 howpublished = {Preprint, {arXiv}:2503.05940 [math.{RT}] (2025)},
 url = {https://arxiv.org/abs/2503.05940},
 arXiv = {arXiv:2503.05940}
}

@article {Deodhar,
    AUTHOR = {Deodhar, Vinay V.},
     TITLE = {A note on subgroups generated by reflections in {C}oxeter
              groups},
   JOURNAL = {Arch. Math. (Basel)},
  FJOURNAL = {Archiv der Mathematik},
    VOLUME = {53},
      YEAR = {1989},
    NUMBER = {6},
     PAGES = {543--546},
      ISSN = {0003-889X,1420-8938},
   MRCLASS = {20H15 (20F55)},
  MRNUMBER = {1023969},
MRREVIEWER = {Fran\c{c}ois\ Digne},
       DOI = {10.1007/BF01199813},
       URL = {https://doi.org/10.1007/BF01199813},
}

@article {quiverf1,
    AUTHOR = {Szczesny, Matt},
     TITLE = {Representations of quivers over {$\mathbb{F}_1$} and {H}all
              algebras},
   JOURNAL = {Int. Math. Res. Not. IMRN},
  FJOURNAL = {International Mathematics Research Notices. IMRN},
      YEAR = {2012},
     PAGES = {2377--2404},
      ISSN = {1073-7928,1687-0247},
   MRCLASS = {16G20 (17B35)},
  MRNUMBER = {2923170},
MRREVIEWER = {Kostiantyn\ Iusenko},
       DOI = {10.1093/imrn/rnr113},
       URL = {https://doi.org/10.1093/imrn/rnr113},
}

@incollection{PRO,
  author    = {Dyckerhoff, Tobias},
  title     = {Higher Categorical Aspects of {H}all Algebras},
  booktitle = {Building Bridges Between Algebra and Topology},
  publisher = {Springer International Publishing},
  year      = {2018},
  pages     = {1--61},
  doi       = {10.1007/978-3-319-70157-8_1}
}

@article{BC,
  author    = {Bergeron, Nantel and Ceballos, Cesar},
  title     = {A {H}opf Algebra of Subword Complexes},
  journal   = {Adv. Math.},
  volume    = {305},
  year      = {2017},
  pages     = {1163--1201},
  doi       = {10.1016/j.aim.2016.10.007},
  mrnumber  = {3570156}
}

@article{Dyer,
  author    = {Dyer, M. J.},
  title     = {$n$-Low Elements and Maximal Rank $k$ Reflection Subgroups of {C}oxeter Groups},
  journal   = {J. Algebra},
  volume    = {607},
  year      = {2022},
  number    = {B},
  pages     = {139--180},
  doi       = {10.1016/j.jalgebra.2021.02.015},
  mrnumber  = {4441316}
}

@article{BDKT,
  author    = {Baumann, Pierre and Dunlap, Thomas and Kamnitzer, Joel and Tingley, Peter},
  title     = {Rank 2 Affine {MV} Polytopes},
  journal   = {Represent. Theory},
  volume    = {17},
  year      = {2013},
  pages     = {442--468},
  mrnumber  = {3072875}
}

@article{MVHD,
  author    = {Baumann, Pierre and Kamnitzer, Joel and Knutson, Allen},
  title     = {The {Mirkovi{\'c}--Vilonen} Basis and {Duistermaat--Heckman} Measures},
  note      = {With an appendix by Anne Dranowski, Kamnitzer and Calder Morton-Ferguson},
  journal   = {Acta Math.},
  volume    = {227},
  year      = {2021},
  number    = {1},
  pages     = {1--101},
  doi       = {10.4310/ACTA.2021.v227.n1.a1},
  mrnumber  = {4346265}
}

@article{GL,
  author    = {Gaussent, St{\'e}phane and Littelmann, Peter},
  title     = {LS Galleries, the Path Model, and {MV} Cycles},
  journal   = {Duke Math. J.},
  volume    = {127},
  year      = {2005},
  number    = {1},
  pages     = {35--88},
  doi       = {10.1215/S0012-7094-04-12712-5},
  mrnumber  = {2126496}
}

@article{JOEL,
  author    = {Kamnitzer, Joel},
  title     = {Mirkovi{\'c}--{V}ilonen Cycles and Polytopes},
  journal   = {Ann. of Math. (2)},
  volume    = {171},
  year      = {2010},
  number    = {1},
  pages     = {245--294},
  doi       = {10.4007/annals.2010.171.245},
  mrnumber  = {2630039}
}

@book{BAI,
  author    = {Bai, Yuguang},
  title     = {Cluster Structure for {M}irkovic--{V}ilonen Cycles and Polytopes},
  note      = {Thesis (Ph.D.)--University of Toronto (Canada)},
  publisher = {ProQuest LLC, Ann Arbor, MI},
  year      = {2022},
  pages     = {94},
  isbn      = {979-8209-90374-1},
  mrnumber  = {4419705}
}

@article{subword,
  author    = {Knutson, Allen and Miller, Ezra},
  title     = {Subword Complexes in {C}oxeter Groups},
  journal   = {Adv. Math.},
  volume    = {184},
  year      = {2004},
  number    = {1},
  pages     = {161--176},
  doi       = {10.1016/S0001-8708(03)00142-7},
  mrnumber  = {2047852}
}

\end{document}